\documentclass[oneside,11pt]{article}
\usepackage[margin=2cm]{geometry} 
\usepackage{amsthm, amsmath, amssymb}
\usepackage[mathlines]{lineno}
\usepackage[dvipsnames]{xcolor}
\usepackage[pagebackref,colorlinks,citecolor=Plum,urlcolor=Periwinkle,linkcolor=DarkOrchid]{hyperref}
\usepackage{graphicx}
\usepackage{enumerate}
\usepackage{nicefrac}
\usepackage[normalem]{ulem}

\def\ba#1\ea{\begin{linenomath}\begin{align*}#1\end{align*}\end{linenomath}}
\def\ban#1\ean{\begin{linenomath}\begin{align}#1\end{align}\end{linenomath}}

\numberwithin{equation}{section}
\newtheorem{theorem}{Theorem}[section]
\newtheorem{lemma}[theorem]{Lemma}

\newtheorem{proposition}[theorem]{Proposition}

\theoremstyle{definition}
\newtheorem{definition}[theorem]{Definition}
\newtheorem*{remark}{Remark}

\title{Catalytic branching random walk}
\author{C\'ecile Mailler\thanks{Department of Mathematical Sciences, University of Bath, Claverton Down, BA2 7AY Bath, UK.\newline Email: \texttt{c.mailler@bath.ac.uk}} \and Bruno Schapira\thanks{Aix-Marseille Universit\'e, CNRS, Centrale Marseille, I2M, UMR 7373, 13453 Marseille, France.\newline Email: \texttt{bruno.schapira@univ-amu.fr}}}

\newcommand{\sss}{\ensuremath{\scriptscriptstyle}}

\definecolor{darkred}{rgb}{0.9,0.1,0.1}

\begin{document}
\maketitle

\begin{abstract}
We show the existence of a phase transition between a localisation and a non-localisation regime for a branching random walk with a catalyst at the origin. More precisely, we consider a continuous-time branching random walk that jumps at rate one, with simple random walk jumps on $\mathbb Z^d$, and that branches (with binary branching) at rate $\lambda>0$ everywhere, except at the origin, where it branches at rate $\lambda_0>\lambda$. We show that, if $\lambda_0$ is large enough, then the occupation measure of the branching random walk localises (i.e.~when normalised by the total number of particles, it converges almost surely without spatial renormalisation), whereas, if $\lambda_0$ is close enough to $\lambda$, then the occupation measure delocalises, in the sense that the proportion of particles in any finite given set converges almost surely to zero.
The case $\lambda = 0$ (when branching only occurs at the origin) has been extensively studied in the literature and a transition between localisation and non-localisation was also exhibited in this case. 
Interestingly, the transition that we observe, conjecture, and partially prove in this paper occurs at the same threshold as in the case~$\lambda=0$.
{One of the strengths of our result is that, in the localisation regime, we are able to prove convergence of the occupation measure, whilst existing results in the case $\lambda = 0$ give convergence of moments instead.}
\end{abstract}

\section{Introduction and main results}
Branching random walks are a simple probabilistic model for populations that move through space. 
In the classical continuous time model, particles all branch at the same rate and between branching events, move according to some Markov process (eg. a simple random walk), independently of each other.
A natural generalisation of this model is to allow the branching rate to depend on the position of the particle: from an application's point of view, one can think of some sites being more favourable to a growth in population than others (for example if there is more food available at those sites).
In some models such as the parabolic Anderson model, in which the branching rates are sampled in an i.i.d.\ way, some sites are sometimes so attractive compared to others that the population localises at these sites or near these sites.
In this paper, we show that localisation can occur in a very simple model where the branching rate is the same everywhere, except at the origin, where it is higher.

We consider the following continuous-time branching random walk on $\mathbb Z^d$ ($d\geq 1$):
at time zero, there is one particle alive in the system, and its position is the origin.
Each particle carries two independent Poisson clocks: the ``jump''-clock that rings at rate~$1$, and the ``branch''-clock that rings at rate~$\lambda_0$ when the particle is at the origin and at rate $\lambda$ everywhere else.
When the jump-clock of a particle rings, then this particle moves to a neighbouring site chosen uniformly at random among the $2d$ neighbouring sites. When the branch-clock of a particle rings, it gives birth to a new particle at the same location. 
This model is a variant of the so-called ``catalytic branching random walk'', which corresponds to taking $\lambda = 0$, and which has been extensively studied in the literature (we give a literature review on this model in Section~\ref{sec:discussion}).

We are interested in the occupation measure of the process, i.e.\ the process $(\Pi_t)_{t\geq 0}$ where, for all~$t\geq 0$,
\[\Pi_t = \sum_{i=1}^{N_t} \delta_{X_i(t)},\]
with $N_t$ the number of particles alive at time $t$, 
and $(X_i(t))_{1\leq i\leq N_t}$ their respective positions.
We assume in the whole paper that $\lambda_0\ge \lambda >0$ (in fact most of our results assume $\lambda_0>\lambda>0$, but we also occasionally discuss the case $\lambda_0 = \lambda$).

\subsection{Results}\label{subsec:results}
Our main result is as follows:
\begin{theorem}\label{th:main}
If $\lambda_0>2d-1+2d\lambda$, then there exists a deterministic probability measure~$\nu$ on $\mathbb Z^d$ such that, almost surely as $t\uparrow\infty$,
\[\widehat \Pi_t := \frac{\Pi_t}{N_t} \to \nu,\]
for the topology of weak convergence on the space of probability measures on $\mathbb Z^d$.
\end{theorem}
The strength of this result is that it gives almost sure convergence of the occupation measure of the branching random walk. In comparison, existing results in the case $\lambda = 0$, such as in~\cite{ABY98, DR}, only give convergence of all moments (see Section~\ref{sec:discussion} for a more detailed discussion and more references on the case $\lambda = 0$). 
Another notable difference is that we use a random normalisation, namely we divide the empirical distribution of particles by~$N_t$, and our limit is a deterministic measure. 
This contrasts for instance with limiting results in the case $\lambda = 0$, such as in D\"oring and Roberts~\cite{DR} or Watanabe~\cite{W67} in a continuous setting, where the normalisation is deterministic and the limiting measure is multiplied by a random factor.
We believe that our results (Theorem~\ref{th:main} and Lemma~\ref{lem:balance1} below) may imply convergence of 
$\mathrm e^{-(\lambda+(\lambda_0-\lambda)\nu_0)t}\Pi_t$.
Proving this would require to estimate the speed of convergence in Theorem~\ref{th:main}, which we have not tried to do so far.

\medskip
We believe that the condition $\lambda_0>2d-1+2d\lambda$ is not optimal. 
To support that claim we first state the following proposition:
\begin{proposition}\label{prop:balance}
Let $(e_1, \ldots, e_d)$ be the canonical basis of $\mathbb R^d$.
If $\widehat\Pi_t \to \nu$ almost surely for the topology of weak convergence, then $\nu$ satisfies
\begin{equation}
\label{eq:balance1}
(1+(\lambda_0-\lambda)\nu_0)\nu_x = \frac1{2d}\sum_{i=1}^d (\nu_{x+e_i}+\nu_{x-e_i}) ,\quad (\forall x\neq0)
\end{equation}
and
\begin{equation}\label{eq:balance2}
(1-(\lambda_0-\lambda)(1-\nu_0))\nu_0= \frac1{2d}\sum_{i=1}^d (\nu_{e_i}+\nu_{-e_i}).
\end{equation}
\end{proposition}
In fact,~\eqref{eq:balance2} is redundant. Indeed, by summing over $x\neq 0$ the left and right hand sides of the equalities in~\eqref{eq:balance1}, 
it is readily seen that any probability measure satisfying~\eqref{eq:balance1} 
also satisfies~\eqref{eq:balance2}. 
Our next result identifies a phase transition for the existence of solutions to the balance equations~\eqref{eq:balance1} and~\eqref{eq:balance2}, 
which we conjecture to coincide with the transition between a localisation and a delocalisation regimes.

\begin{theorem}\label{th:no_loc} Let $\gamma_d$ be the probability that a random walk in $\mathbb Z^d$ started at zero never returns to zero.
\begin{itemize}
\item[(i)] If  $0\le \lambda_0-\lambda\le \gamma_d$, then there exist no probability measure satisfying \eqref{eq:balance1}.
\item[(ii)] If $\lambda_0-\lambda> \gamma_d$, then there exists a unique probability measure satisfying \eqref{eq:balance1}.
\end{itemize}
\end{theorem}
In other words, we prove that, ({\sc a}) if $\lambda_0>2d-1+2d\lambda$, then the occupation measure of the branching random walk ``localises'' (see Theorem~\ref{th:main}), and ({\sc b}) if $0\le \lambda_0-\lambda\le \gamma_d$, then localisation cannot occur in the strong sense of ({\sc a}) (see Theorem~\ref{th:no_loc}$(i)$). We stress that the value of the threshold for the transition observed in Theorem~\ref{th:no_loc}, namely the constant $\gamma_d$, is exactly the same as in the case $\lambda=0$, see below. 
However, while the phase transition is elementary to identify in the latter case, using a direct comparison between the total number of particles ever visiting the origin and a standard Galton-Watson process (see Remark~\ref{rem:subcritical} below), it seems less immediate to use such a comparison when $\lambda$ is positive. 
Actually, it may even seem counter-intuitive that the threshold for the phase transition between localisation and delocalisation holds at the same critical value (for the parameter $\lambda_0-\lambda$), no matter the value of $\lambda$.
Indeed, one could expect that increasing $\lambda$ would favour the natural diffusive behaviour of the process, while our results tend to show that, on the contrary, it does not play any role in this respect. 

\medskip
Our last main result shows a weak form of localisation, as soon as $\lambda_0>\lambda+1$, in the sense that almost surely at all large times, a positive proportion of all particles stands at the origin. It also shows that when $\lambda_0<\lambda + \gamma_d$, almost surely the proportion of particles that stand at the origin (or in any given finite set) converges to zero as time goes to infinity. 

From now on, in a slight abuse of notation, we write $\widehat \Pi_t(x)$ instead of $\widehat\Pi_t(\{x\})$, for all $x\in\mathbb Z^d$.

\begin{theorem} \label{prop:weakloc}
\begin{enumerate}
\item[(i)] If $\lambda_0>1+\lambda$, then, almost surely, 
\[\liminf_{t\to \infty} \ \widehat {\Pi}_t(0) \ge \frac{\lambda_0 - \lambda - 1}{\lambda_0 -\lambda}.\]
\item[(ii)] If $\lambda <  \lambda_0 <\lambda + \gamma_d$, then for any finite $B\subset \mathbb Z^d$, almost surely, 
$$\lim_{t\to \infty} \widehat \Pi_t(B) =0. $$ 
\end{enumerate}
\end{theorem}
In addition to being interesting in itself, the first part of this theorem is used as a preliminary step in the proof of Theorem~\ref{th:main}, see Remark~\ref{remark:A'1}.
Concerning the second part, we note that when $\lambda_0 = \lambda$, it is well-known\footnote{
We were unable to find a reference where this is done in our exact setting, but one can apply standard martingale techniques as in~\cite{Biggins92} to get this result.} that, almost surely as $t\uparrow\infty$, 
\[\widehat\Pi_t\big(\cdot\sqrt{t}\big) \to \mathcal N(0, \text{Id}),\]
for the topology of weak convergence on the space of probability measures on $\mathbb R^d$, where $\text{Id}$ stands for the identity matrix. 
Based on first moment estimates, we conjecture that such a strong result should persist in the whole subcritical phase, i.e.~for any $\lambda\leq \lambda_0 < \lambda + \gamma_d$.

\begin{remark}
Several of the assumptions we make on the catalytic branching random walk are, 
we believe, not necessary (and are not generally made in the literature on the case $\lambda = 0$).
For example, the assumption that, at time zero, there is a unique particle and this particle sits at the origin could be replaced by having a finite number of particles alive at time zero, each of them at any site in $\mathbb Z^d$. 
This is because of the irreducibility of the process.
Also, we believe that the binary branching could be replaced by any supercritical branching mechanism (even with a random number of offspring).
Finally, the particles could move according to a Markov jump process instead of a continuous time simple symmetric random walk, as assumed here.
Our proofs should be robust under all these modifications, although the proof of Theorem~\ref{th:main} might become more intricate, in particular in the case of a more general underlying Markov jump process. This is the reason why we decided to keep our framework as simple as possible.
\end{remark}

\subsection{Discussion of the proofs (and plan of the paper)}
The proofs of Theorem~\ref{prop:weakloc} (i) and (ii) (see Section~\ref{sec:new}) rely on two different ideas. 
For the proof of the weak localisation result, we first show that the number of particles at the origin grows exponentially fast almost surely. For this we couple our catalytic branching random walk with the classical branching random walk (corresponding to the case $\lambda = \lambda_0$) and use a second moment method. We then use the excess of the branching rate at the origin and Borel-Cantelli type arguments, to bootstrap this estimate, and show that the proportion of particles at the origin remains bounded from below over time by a positive constant.
The proof of the weak delocalisation result relies on a first moment method and on a direct link between the first moment in the case $\lambda_0>\lambda$ and the first moment in the case when branching only occurs at the origin, at rate $\lambda_0-\lambda$, see Proposition~\ref{prop:exp}.

The proof of Proposition~\ref{prop:balance} (see Section~\ref{sec:balance}) relies on (1) finding a good approximation for the total number of particles alive at large times $t$ (see Lemma~\ref{lem:balance1} - this is done by finding a martingale and proving that it converges in $L^2$), which allows to replace the random normalisation of the occupation measure by a deterministic one (see Lemma~\ref{lem:balance2}), and then (2) writing a forward evolution equation for the occupation measure and studying their equilibrium.

The proof of Theorem~\ref{th:no_loc} (see Section~\ref{sec:noloc}) is done by studying the existence of solutions (or lack of thereof) to the balance equations of Proposition~\ref{prop:balance}.
To do this, we use the optional stopping theorem applied to a martingale that naturally arises from the balance equations.

The proof of localisation (see Section~\ref{sec:mvpp}), i.e.\ Theorem~\ref{th:main}, relies on proving that the occupation measure taken at the times when either a branching or a jumping event occurs is a measure-valued P\'olya process ({\sc mvpp}).
{\sc Mvpp}s were introduced in Bandyopadhyay and Thacker~\cite{BT16} and Mailler and Marckert~\cite{MM17}  (see also~\cite{BT17} where the model was first defined on a particular example) as a generalisation of P\'olya urns to infinitely-many colours.
Mailler and Marckert~\cite{MM17} proved convergence in probability of a large class of {\sc mvpp}s under an assumption of balance (in terms of urns, the balance assumption requires that the total number of balls added at any time-step in the urn is constant and deterministic). 
They also required that the replacement rule is deterministic, but this assumption was later removed by Janson~\cite{J19}.
To prove Theorem~\ref{th:main}, we use convergence results proved by Mailler and Villemonais~\cite{MV20}, which apply without the balance assumption, for random replacements, and hold almost surely.
Besides showing that our model is an {\sc mvpp}, our main contribution is to check that our {\sc mvpp} satisfies the assumptions of~\cite{MV20}; to do this, we need to show that a continuous-time sub-Markovian process admits a quasi-stationary distribution, and this is done using results of Champagnat and Villemonais~\cite{CV23}.
Interestingly, this quasi-stationary distribution is the limit $\nu$ in Theorem~\ref{th:main}, and our Proposition~\ref{prop:balance} provides a characterisation of this measure for $\lambda_0$ large enough.

As already discussed, we believe that the condition $\lambda_0>2d-1+2d\lambda$ is far from optimal (and conjecture that there is localisation as soon as $\lambda_0 >\lambda+\gamma_d$). 
The reason for this is that the proof of Theorem~\ref{th:main} relies on~\cite[Theorem~1]{MV20} and~\cite[Theorem~5.1]{CV23}, 
both of which are of the form ``sufficient assumptions imply result'' as opposed to ``necessary and sufficient assumptions imply result''.
To improve our condition on $\lambda_0$, one would have to come up with an ad-hoc proof of these two theorems for the catalytic branching random walk, or come up with a completely different proof that does not involve MVPPs.

\subsection{Discussion of the literature}\label{sec:discussion}

\textbf{The catalytic branching random walk.} 
The case when branching only happens at the origin (thus $\lambda_0>\lambda = 0$) 
has been long studied in the literature, under the name of catalytic branching random walk.
As far as we know, 
most of the results concerning the localisation/delocalisation of the process are about the asymptotic 
behaviour of the moments of the number of particles at each site as time goes to infinity.

The first papers on the catalytic branching random walk focused on particles performing simple random walks on $\mathbb Z^d$: see, e.g., Albeverio, Bogachev and Yarovaya~\cite{ABY98}.
Efforts were later made to generalise these results allowing the particles to perform more general Markovian trajectories (between branching events). Several methods were used to analyse the aforementioned moments: Bellman-Harris branching processes in~\cite{TV03, VT05, B10, B11}, operator theory in~\cite{Y10}, renewal theory in~\cite{HVT10}, and spine decomposition techniques in D\"oring and Roberts~\cite{DR}.

More recently, some of the focus has shifted to understanding the spread of the catalytic branching random walk at large times see, e.g.~\cite{CH14, MY12, B18, B20}. 
To the best of our knowledge, this interesting question is open in our setting (i.e.\ with branching outside the origin).

Most of these papers consider a more general branching mechanism than our binary fission: the offspring distribution can be any distribution with finite second moment. They also consider the case when the offspring distribution is (sub-)critical (i.e.\ the average number of offspring is less than or equal to~$1$). 
In this paper, we focus on a particular super-critical case. However, we allow arbitrary rates $\lambda$ and~$\lambda_0$ for the branching events, which, at least phenomenologically, plays a similar role as allowing a general average number of offspring.

To compare the aforementioned results with ours, we now summarise the results of, e.g.\ \cite{ABY98} (see also~\cite{DR} for similar results in a slightly different model). 
For all $y\in\mathbb Z^d$, $t\geq 0$, and $k\geq 1$, we let $M^k(t, y)$ be the $k$-th moment of the number of particles at position $y$ at time $t$, and $M^k(t)$ be the $k$-th moment of the total number of particles in the system at time~$t$.
Albeverio, Bogachev and Yarovaya~\cite{ABY98} prove that:
\begin{itemize}
\item If $\lambda_0<\gamma_d$, then, there exist positive constants $C, (C(y))_{y\in\mathbb Z^d}$ such that,
for all $y\in\mathbb Z^d$, as $t\to+\infty$,
\begin{equation}\label{eq:ABY}
M^1(t, y)=(C(y)+o(1)) t^{-\nicefrac d2}
\quad\text{ and }\quad
M^1(t) = C + o(1),
\end{equation}
which corresponds to a non-localisation phase as most particles in the system have drifted away from the origin.
\item If $\lambda_0>\gamma_d$, then, there exist positive constants $\rho>0$, $C$, $(C(y))_{y\in\mathbb Z^d}$ such that, for all $y\in\mathbb Z^d$, for all $k\geq 1$, as $t\to+\infty$,
\[M^{k}(t, y) = (C(y)+o(1)) \mathrm e^{k\rho t} \quad
\text{ and }\quad
M^{k}(t)=(C+o(1)) \mathrm e^{k\rho t},
\]
which corresponds to a localisation phase, as the number of particles at any site $y$ is of the same order as the total number of particles in the system. 
\end{itemize}
Note that the critical case, when $\lambda_0 = \gamma_d$, is also studied in~\cite{ABY98}, and the growth rate of the number of particles in this case is sub-exponential (see also Vatutin and Topchii~\cite{VT11}).

We believe that, for most of our results, the assumption that $\lambda$ is positive could be dropped. 
In particular, when $\lambda = 0$, the convergence result of Theorem~\ref{th:main} should hold almost surely on the event that the number of particles grows to infinity. 
The main challenge would be to prove an analogue of Lemma~\ref{lem:2.1} below, 
showing that on the event of non-local extinction, 
the number of particles at the origin grows almost surely exponentially fast. 
While it could very well be that our proof of Lemma~\ref{lem:2.1} could be adapted to this case, using in particular the asymptotics of the first two moments of $\Pi_t(0)$ (proved in~\cite{ABY98}), we have not tried to pursue in this direction.

\medskip
{\bf Generalisations of the catalytic branching random walk:}
Since its introduction, the catalytic branching random walk (with branching only at the origin) has been generalised in several ways. One direction of generalisation is to add more catalysts (sites at which branching happens): instead of having one catalyst at the origin, one can have $N$ catalysts at positions $x_1, \ldots, x_N$ (see, e.g. Albeverio and Bogachev~\cite{AB00}).
Another direction is to allow the catalysts to move following their own Markovian motion and branching only happens for particles at the same position as one of the catalysts, {with a rate that may depend on the number of catalysts there (see, e.g.~\cite{GH06, KS03, CGM12}). 
{Finally let us mention the model of inhomogeneous branching random walk, where the branching rate depends on the position of the particle, as e.g.~in~\cite{BH14a,GKW99}.}

\medskip
{\bf Multi-type Galton-Watson processes:} A natural discrete-time version of our process is the following. Start with one particle at the origin say, at time $0$. Then given two probability measures $\mu_0$ and $\mu$, at each unit of time, particles alive at that time give birth to a number of offspring distributed respectively as $\mu_0$ if at the origin or as $\mu$ if not at the origin, and die immediately after that, whilst new born particles immediately jump at random to a nearest neighbor. Interpreting the spatial positions of particles as their types, we can also view this process as a multi-type Galton-Watson process. There are now a number of deep results regarding the asymptotic behavior of the empirical distribution of types for such processes, see in particular~\cite{AD25} and~\cite{BBC25}, and references therein. It would be interesting to investigate if for such process, a sharp phase transition could be derived, depending on the means of $\mu_0$ and $\mu$. 

\medskip
{\bf Catalytic branching Brownian motion/superprocesses, mutually catalytic processes:}
Various models of catalytic branching Brownian motions or superprocesses have also been considered. 
The case of inhomogeneous branching Brownian motion, where the branching rate depends continuously on the position, has received quite some attention, starting with Watanabe~\cite{W67}, who proves a similar result as our Theorem~\ref{th:main}, 
yet with a different (and deterministic) normalisation and a random limit.
More recent results concern the position of the right-most particle in dimension one (see, e.g., Lalley and Sellke~\cite{LS88,LS89}).  
The question of localisation as well as other questions are also addressed in a model of catalytic branching Brownian motion where branching only occurs at the origin in Bocharov and Harris~\cite{BH14,BH16}. 
Finally, catalytic or mutually catalytic superprocesses have been extensively studied, e.g.~in~\cite{CDG04,DF94,FLG95}.

\medskip
{{\bf Branching random walks in random environment and parabolic Anderson model:} 
To conclude this discussion, 
let us briefly mention the model of branching random walks in random environment and their mean-field counterpart called the parabolic Anderson model.
In this model, first introduced by G{\"a}rtner and Molchanov~\cite{GM90} one considers particles that move according to a continuous time simple random walk on $\mathbb Z^d$ and branch at rate $\xi_x$ when they are at position $x$, where the sequence $(\xi_x)_{x\in\mathbb Z^d}$ is a sequence of i.i.d.\ random variables called potentials.
It is conjectured that, in this model, the occupation measure at large times is concentrated in a few ``peaks'' (this phenomenon is called ``intermittency''). 
This is a difficult problem, and only partial progress has been made towards solving it so far: one important result of Ortgiese and Roberts~\cite{OR17} is that, when the potentials are Pareto-distributed, there is one peak (this phenomenon is called ``strong intermittency'').
The mean-field version of this branching random walk in random environment is the parabolic Anderson model, which is much better understood. In particular, in this setting a phenomenon of strong intermittency also arises (see, e.g., G\"artner and den Hollander~\cite{GdH06},  K\"onig, Lacoin, M\"orters and Sidorova~\cite{KLMS09}, Fiodorov and Muirhead~\cite{FM14}).}


\section{Proof of Theorem~\ref{prop:weakloc}}\label{sec:new}
In this section we prove Theorem~\ref{prop:weakloc}. Recall that the first part states a weak localisation result, namely that almost surely a positive proportion of the particles stands at the origin at all large times, when $\lambda_0>\lambda + 1$, while the second part proves a weak form of delocalisation in the whole subcritical phase, namely that for all $\lambda < \lambda_0<\lambda + \gamma_d$, almost surely the proportion of particles that stand in any finite set converges to zero as time goes to infinity. We prove the first part in Subsection~\ref{subsec:weakloc} and the second part in Subsection~\ref{subsec:weakdeloc}.

\subsection{Proof of Part (i) (weak localisation result)}\label{subsec:weakloc}
We start with a basic lemma ensuring that almost surely the number of particles at the origin grows sufficiently fast to infinity. 
\begin{lemma} \label{lem:2.1}
Assume $\lambda_0\ge \lambda>0$. 
Then almost surely, for any $\alpha \in [0,\lambda)$, 
\[\lim_{t\to \infty} \mathrm e^{-\alpha t} \cdot \Pi_t(0) = \infty.\]
\end{lemma}
\begin{proof}
Let $\Gamma_t$ be the occupation measure of the branching random walk in which all particles jump at rate~1 and branch at rate~$\lambda$ (in other words, $\Gamma_t$ is $\Pi_t$ in the case when $\lambda = \lambda_0$). One can couple both branching random walks such that $\Gamma_t(0)\leq \Pi_t(0)$ almost surely for all $t\geq 0$ (indeed, the extra branching rate at zero can only increase the number of particles at each state).
It is thus enough to prove the result in the case $\lambda = \lambda_0$, which we assume now. 
Although we believe this result is folklore, we cannot find an exact reference (see Revesz~\cite[Theorem~4.5]{Revesz} for a local limit theorem on the analogous discrete-time branching random walk) and thus provide a proof for completeness' sake. We prove this by a second moment method: first note that
\[\mathbb E[\Pi_t(0)] 
= \mathbb E\bigg[\sum_{i=1}^{N_t} {\bf 1}_{X_i(t) = 0}\bigg]
= \mathbb E[N_t] \cdot \mathbb P(X(t) = 0),\]
where $(X(t))_{t\geq 0}$ is a simple symmetric random walk on $\mathbb Z^d$ with jump rate~$1$, and the $(X_i(t))_{i\le N_t}$ are the positions of the particles of our branching random walk at time $t$.
This is true because, when $\lambda_0 = \lambda$, $(N_t)_{t\geq 0}$ is a Yule process of parameter $\lambda$, independent of each $X_i(t)$, for any fixed $i\le N_t$. Thus,
\begin{equation*}
\mathbb E[\Pi_t(0)] = \mathrm e^{\lambda t} \cdot \mathbb P(X(t)=0).
\end{equation*}
Using the local central limit theorem (see, e.g., \cite[Theorem~2.5.6]{LL10} for the one-dimensional case, which implies the $d$-dimensional case as the $d$ coordinates are independent), 
we then get, for some constant $c>0$, 
\begin{equation*}
\mathbb E[\Pi_t(0)] \ge \frac{c\cdot \mathrm  e^{\lambda t}}{t^{d/2}}. 
\end{equation*}
For the second moment, we simply use that $N_t$ is a geometric random variable with parameter $\mathrm e^{-\lambda t}$ (see~\cite[Section~III.5]{AK}), implying that $\mathbb E[N_t^2] = (2+o(1)) \mathrm e^{2\lambda t}$, as $t\to \infty$, and consequently also  
\begin{equation}\label{eq:Ntsquare}
\mathbb E[\Pi_t(0)^2] \le \mathbb E[N_t^2] \le 3 \mathrm  e^{2\lambda t}, 
\end{equation}
for all $t$ large enough.  
Fix $\alpha<\lambda$. Applying Paley-Zygmund's inequality, we deduce from the previous bounds that,
 for all $t$ large enough, for some possibly smaller constant $c>0$, 
\[\mathbb P(\Pi_t(0)\ge \mathrm e^{\alpha t}) \ge \frac{c}{t^d}.\]
In fact a similar reasoning shows that, for all $t$ large enough, 
\begin{equation}\label{eq.Pit.1}
\inf_{\|x\|\le t^{1/4}} \mathbb P(\Pi_t(x)\ge \mathrm  e^{\alpha t} ) \ge \frac{c}{t^d}, 
\end{equation}
since we also have $\inf_{\|x\|\le t^{1/4}} \mathbb P(X(t)=x)\ge c/t^{d/2}$ (again, see \cite[Theorem~2.5.6]{LL10}). 
Recall next that for some positive constants $c'$ and $C$, for all $T>0$, 
\[\mathbb P( \|X(T)\| \ge T) \le  C\exp(-c'T).\] 
Hence Markov's inequality yields
\begin{align}
\mathbb P\Big(\sum_{\|x\|\le t^{1/4}} \Pi_{t^{1/4}}(x) \ge \frac 12 N_{t^{1/4}}\Big) 
& = \mathbb P\Big(\big|\{i\le N_{t^{1/4}} : \|X_i(t^{1/4})\|\le t^{1/4}\}\big|\ge \frac 12 N_{t^{1/4}}\Big)   \notag\\
& \ge 1 -C \exp(-c't^{1/4}), \label{eq.Pit.2} 
\end{align} 
Combining~\eqref{eq.Pit.1} and~\eqref{eq.Pit.2} yields, by first conditioning on the positions of the particles at time $t^{1/4}$, that for all $t$ large enough, 
\begin{align*}
 \mathbb P(\Pi_t(0)\ge e^{\alpha (t-t^{1/4})} )&  \ge 1 - C\exp(-c't^{1/4})- \mathbb E\Big[(1-\frac{c}{t^d})^{\frac 12 N_{t^{1/4}}}\Big]  \\
 & \ge 1 - 2C\exp(-\kappa\cdot t^{1/4}),
 \end{align*}
for some other constant $\kappa>0$, 
Then using Borel-Cantelli's lemma, we deduce that for any $\alpha'<\alpha$, almost surely, 
$$ \liminf_{n\to \infty} \ \Pi_n(0) \cdot \mathrm e^{-\alpha' n} = +\infty.$$
Note also that any particle that stands at the origin at some time $n$ will stay there until time $n+1$ 
with probability $1/\mathrm e$. 
Hence, by standard large deviations estimates and Borel-Cantelli lemma again, we get that, almost surely for any $\alpha''<\alpha'$, 
\[\liminf_{t\to \infty} \ \Pi_t(0) \cdot \mathrm e^{-\alpha'' t} = +\infty.\]
The fact that this holds for any $\alpha''<\alpha'<\alpha <\lambda$ implies the result.  
\end{proof}

We next record a simple fact:
\begin{lemma}\label{lem:2.2}
Assume $\lambda_0>1$ and let $\delta\in (0,\lambda_0- 1)$. There exists $s_0>0$, such that for any $s\in [0,s_0]$, 
$$\mathbb E[\Pi_s(0)] \ge \mathrm e^{(\lambda_0-1-\delta)s}. $$ 
\end{lemma}
\begin{proof} 
The result directly follows from the observation that, by definition of our catalytic branching random walk as a Markov jump process starting with one particle at the origin, as $s\to 0$,
\[\mathbb E[\Pi_s(0)] = 1+ (\lambda_0-1) s + \mathcal O(s^2).\qedhere\]
\end{proof}
As a corollary of the two previous results one can show the following: 
\begin{lemma}\label{lem:2.3}
Assume $\lambda_0>1$ and let $\delta \in (0,\lambda_0-1)$. Then, with $s_0$ as in Lemma~\ref{lem:2.2}, for all $s\in [0,s_0]$, almost surely, 
\[\liminf_{n\to \infty} \frac{\Pi_{(n+1)s}(0)}{\Pi_{ns}(0)} \ge \mathrm e^{(\lambda_0 - 1-\delta)s},\]
and 
\[\liminf_{n\to \infty}  \inf_{u\in [0,s]} \frac{\Pi_{ns+u}(0)}{\Pi_{ns}(0)} \ge 1 - s. \]
\end{lemma}
\begin{proof}
Note that, for all $n\ge 1$ and $s\in [0,s_0]$,
\[\frac{\Pi_{(n+1)s}(0)}{\Pi_{ns}(0)} \ge \frac{1}{\Pi_{ns}(0)} \sum_{i=1}^{\Pi_{ns}(0)} \Pi_s^{(i)}(0),\]
where $(\Pi_s^{(i)}(0))_{i\ge 1}$ is a sequence independent copies of~$\Pi_s(0)$. 
Hence, by Chebyshev's inequality, we get that, for any $\varepsilon \in (0,1)$, for some constant $C>0$,  
\[\mathbb P\bigg(\frac{\Pi_{(n+1)s}(0)}{ \Pi_{ns}(0)} \le (1-\varepsilon) \cdot \mathbb E[\Pi_s(0)] \ \Big|\ \mathcal F_{ns}\bigg) \le \frac{C}{\Pi_{ns}(0)},\]
where $(\mathcal F_t)_{t\ge 0}$ is the natural filtration of the process $(\Pi_t)_{t\ge 0}$. Using now Lemma~\ref{lem:2.1}, and the conditional Borel-Cantelli lemma, this proves that for any fixed $\varepsilon	>0$, almost surely, 
\[\liminf_{n\to \infty} \frac{\Pi_{(n+1)s}(0)}{\Pi_{ns}(0)} \ge (1-\varepsilon) \cdot \mathbb E[\Pi_s(0)].\]
Together with Lemma~\ref{lem:2.2}, this proves the first part of the lemma. 
The second part is similar: We first note that, for any fixed $n\ge 1$, 
$$\inf_{u\in [0,s]} \frac{\Pi_{ns+u}(0)}{\Pi_{ns}(0)} \ge \frac{1}{\Pi_{ns}(0)} \sum_{i=1}^{\Pi_{ns}(0)} \inf_{u\in [0,s]} \Pi_u^{(i)}(0).$$
If the particle at the origin at time~0 does not jump before time $s$, which occurs with probability $\mathrm e^{-s}$, then $\inf_{u\in [0,s]} \Pi_u^{(i)}(0)\ge 1$. Therefore,
\[\mathbb E\Big[\inf_{u\in [0,s]} \Pi_u(0)\Big] \ge \mathrm e^{-s}\ge 1-s.\]
We conclude the proof with similar arguments as in the first part of the proof.
\end{proof}
Similarly as above, we can show the following fact:
\begin{lemma}\label{lem:2.4} 
Let $\delta>0$. There exists $s_1>0$ such that, for any $s\in  [0,s_1]$, almost surely, 
\[\limsup_{n\to \infty} \frac{N_{(n+1)s}}{N_{ns}}\mathrm e^{-[\lambda + \delta + (\lambda_0-\lambda) \widehat \Pi_{ns}(0)] \cdot s} \le 1.\]
\end{lemma}
\begin{proof}
We use that for any $s,t\ge 0$, 
$$\frac{N_{t+s}}{N_t} \le  \frac 1{N_t} \sum_{i=1}^{\Pi_t(0)} N^{(i)}_s + \frac 1{N_t}\sum_{i = 1}^{N_t - \Pi_t(0)} \widetilde N^{(i)}_s, $$ 
where $(N^{(i)}_s)_{i\ge 1}$ is a sequence of independent copies of $N_s$ 
and $(\widetilde N^{(i)}_s)_{i\ge 1}$ is a sequence of i.i.d.\ random variables distributed as the number of particles at time~$s$ in the process {which starts with one particle until an Exponential random time with mean $1/(1+\lambda)$, at which the particle either splits into two particles with probability $\lambda/(1+\lambda)$ or stays as it is otherwise, and then all particles just branch at rate $\lambda_0$ independently of each other}.  
The result follows by observing that $\mathbb E[N_s] = 1+ \lambda_0 s + \mathcal O(s^2)$, 
and $\mathbb E[\widetilde N^{(i)}_s] = 1+ \lambda s+\mathcal O(s^2)$, and then using the same argument as in the proof of Lemma~\ref{lem:2.2}.  
\end{proof}
We can now conclude the proof of the localisation part in Theorem~\ref{prop:weakloc}: 
\begin{proof}[Proof of Theorem~\ref{prop:weakloc} (i).]
Let $\delta \in (0,(\lambda_0 - \lambda - 1)/4)$, and let $s_0$ and $s_1$ be as in Lemma~\ref{lem:2.3} and~\ref{lem:2.4} respectively. Using these two lemmas, we infer that for any given $s \in [0,\min(s_0,s_1)]$, almost surely for all $n$ large enough, 
\[\frac{\widehat \Pi_{(n+1)s}(0)}{\widehat \Pi_{ns}(0)} \ge 1+ \Big(\lambda_0 -1 - \lambda - 3\delta - (\lambda_0 - \lambda)\widehat \Pi_{ns}(0)\Big) s  . \]
Hence, for any $n$ large enough, if 
\[\widehat \Pi_{ns}(0) \le \frac{\lambda_0 -1 - \lambda - 4\delta}{\lambda_0-\lambda}, \]
then $\widehat \Pi_{(n+1)s}(0)\ge (1 +\delta s) \cdot \widehat \Pi_{ns}(0)$. Together with the second part of Lemma~\ref{lem:2.3}, this shows that 
\[\liminf_{t\to \infty} \ \widehat \Pi_t(0) \ge (1-s)^2 \cdot \frac{\lambda_0 -1 - \lambda - 4\delta}{\lambda_0-\lambda}. \] 
Since this holds for any $\delta$ and $s$ small enough, we deduce that 
\[\liminf_{t\to \infty}\ \widehat \Pi_t(0) \ge \frac{\lambda_0 -1 - \lambda}{\lambda_0-\lambda},\]
as claimed.
\end{proof}


\subsection{Proof of Part (ii) (weak delocalisation result)}\label{subsec:weakdeloc}
Here we prove the second part of Theorem~\ref{prop:weakloc}. 
The proof is based on a first moment method. 
More precisely we relate the first moment of the occupation measure 
of any set to the same quantity in the model where branching only occurs 
at the origin at rate $\lambda_0-\lambda$, 
which is the model considered for instance in~\cite{ABY98}. 

\begin{proposition}\label{prop:exp}
Let $\lambda_0\geq \lambda>0$. 
Let $(\widetilde \Pi_t)_{t\ge 0}$ be the occupation measure of the process 
where particles only branch at the origin at rate $\lambda_0-\lambda$, 
and independently move as continuous time simple symmetric random walks with jump rate~$1$. 
Then for all $x\in \mathbb Z^d$ and $t\ge 0$, 
\[\mathbb E[\Pi_t(x)] = \mathrm e^{\lambda t} \cdot \mathbb E[\widetilde \Pi_t(x)]. \]
\end{proposition}
Note that this proposition holds for any $\lambda_0\ge \lambda$, in particular it holds as well in the critical regime ($\lambda_0=\lambda + \gamma_d$),  and in the super-critical regime $(\lambda_0>\lambda + \gamma_d$), and this may be used to provide new results in both cases, see Remark~\ref{rem:exp} below.

\begin{proof}[Proof of Proposition~\ref{prop:exp}]
For $t\geq 0$ and $x\in\mathbb Z^d$, let $u(x,t) = \mathbb E[\Pi_t(x)]$. 
By definition, one has that $u$ is $C^1$ in the second variable, and satisfies the partial differential equation: 
\[\frac{\partial u}{\partial t}   =  \Delta u + (\lambda  +  (\lambda_0-\lambda){\bf 1}_{\{x=0\}}) u,\]
where $\Delta$ is the discrete laplacian 
defined by 
$$\Delta f (x) = \frac{ 1}{2d} \sum_{i=1}^d (f(x+ e_i)+f(x-e_i) - 2f(x)),$$ 
for $f:\mathbb Z^d\to [0,+\infty)$, and $x\in\mathbb Z^d$ (recall that $(e_i)_{1\le i\le d}$ denotes the canonical basis of $\mathbb Z^d$). 
This implies that the function $v$ defined by $v(x,t) = \mathrm e^{-\lambda t} u(x,t)$ (for all $x\in\mathbb Z^d$ and $t\geq 0$) satisfies
\begin{equation}\label{eds.v}
\frac{\partial v}{\partial t}   =  \Delta v+ (\lambda_0 -\lambda){\bf 1}_{\{x=0\}} v.
\end{equation}
In other words, $v$ is solution of the same equation as the function $\widetilde v$ defined by $\widetilde v(x,t)= \mathbb E[\widetilde \Pi_t(x)]$, for all $x\in \mathbb Z^d$ and $t\ge 0$. 
Hence, it just remains to prove that regular solutions to~\eqref{eds.v} are unique. 
This is standard, and can be seen as follows in our case:
Let $v_1$ and $v_2$ be two solutions of \eqref{eds.v} that are $C^1$ in the second variable, square integrable in the first variable, and coincide at time $t=0$. Let 
$$h(x,t) = \exp(-(\lambda_0- \lambda)t) (v_1-v_2)(x,t),\quad \textrm{and}\quad  E(t) = \sum_{x\in \mathbb Z^d} h(x,t)^2.$$
Then a direct computation shows that 
\begin{align*}
\frac{\partial }{\partial t}E(t) & = 2\sum_{x\in \mathbb Z^d}  h(x,t)\cdot \frac{\partial}{\partial t} h(x,t) = 2\sum_{x\in \mathbb Z^d} h(x,t) \cdot \Big( \Delta h(x,t) + (\lambda_0-\lambda) h(x,t) ({\bf 1}_{\{x=0\}} - 1) \Big) \\
& \le 2 \sum_{x\in \mathbb Z^d} h(x,t) \Delta h(x,t). 
\end{align*}
Now note that, for any function $h: \mathbb Z^d \to \mathbb R$ 
such that $\sum_{x\in \mathbb Z^d} h(x)^2<\infty$, 
\begin{align*}
\sum_{x\in \mathbb Z^d} h(x) \Delta h(x) & = \frac 1{2d} \sum_{x\in \mathbb Z^d} \sum_{y\sim x} h(x) h(y) - \sum_{x\in\mathbb Z^d} h(x)^2\\
& \le \frac 1{2d} \sum_{x\in \mathbb Z^d} \sum_{y\sim x} \frac{h(x)^2 +  h(y)^2}{2} - \sum_{x\in\mathbb Z^d} h(x)^2 \le 0,
 \end{align*}
 where $y\sim x$ means that $x$ and $y$ are neighbors in $\mathbb Z^d$.
 Thus, $\frac{\partial }{\partial t}E(t) \le 0$, for all $t\ge 0$. 
Since, by definition, $E(0) = 0$ and $E(t)\ge 0$ for all $t\ge 0$, 
we deduce that $E(t) = 0$ for all $t\ge 0$, 
which concludes the proof of the uniqueness of solutions of~\eqref{eds.v}. 

Hence it only remains to show that both $v$ and $\widetilde v$ are square integrable in the first variable 
(the fact that they are $C^1$ in the second variable is standard). 
For this one can write that for any $t\ge 0$, with the notation from the proof of Lemma~\ref{lem:2.1}, 
\begin{align*} 
\sum_{x\in \mathbb Z^d} v^2(x,t) & = \mathrm e^{-2\lambda t} \cdot \sum_{x\in \mathbb Z^d} \mathbb E\bigg[\sum_{i=1}^{N_t} {\bf 1}_{X_i(t) = x}\bigg]^2  
\le \mathrm e^{-2\lambda t} \cdot\sum_{x\in \mathbb Z^d} \mathbb E\bigg[\bigg(\sum_{i=1}^{N_t} {\bf 1}_{X_i(t) = x}\bigg)^{\!\!2}\bigg] \\
& \le \mathrm e^{-2\lambda t} \cdot\sum_{x\in \mathbb Z^d} \mathbb E\bigg[N_t\sum_{i=1}^{N_t} {\bf 1}_{X_i(t) = x}\bigg] 
= \mathrm e^{-2\lambda t} \cdot \mathbb E[N_t^2] 
\le C \mathrm e^{2(\lambda_0 - \lambda)t}<\infty, 
\end{align*}
for some constant $C>0$, using for the last inequality that our catalytic branching process is stochastically dominated by a process branching at rate $\lambda_0$ everywhere, and the same bound as in~\eqref{eq:Ntsquare}. A similar argument works as well for $\widetilde v$, and 
this concludes the proof of the proposition. 
\end{proof}

We now resume the proof of Theorem~\ref{prop:weakloc} (ii). 
Note that we may assume $d\ge 3$, as otherwise $\gamma_d= 0$. 
{Using the same notation as in Proposition~\ref{prop:exp},
we let $\widetilde u(x,t) = \mathbb E[\widetilde \Pi_t(x)]$ 
and $\widetilde N_t$ be the number of particles at time~$t$ in the process that starts with one particle at the origin and only branches at the origin at rate $\lambda_0-\lambda$. 
Thanks to~\eqref{eq:ABY}, one has for some constant $C>0$, and all $t\ge 1$, 
\begin{equation}\label{eq:tilde.u}
\widetilde u(x,t) \le \frac{C}{t^{d/2}}.
\end{equation}
Together with Proposition~\ref{prop:exp}, we get  
\[u(x,t) \le \frac{C}{t^{d/2}} \cdot \mathrm e^{\lambda t}.\]
}%
Since we assumed $d\ge 3$, it follows from Borel-Cantelli lemma and Markov's inequality, that almost surely, 
\begin{equation}\label{subcritic0} 
\limsup_{n\to \infty} \ \mathrm e^{-\lambda n} \cdot \Pi_n(x) = 0,
\end{equation}
for all $x\in \mathbb Z^d$. We claim that we can deduce from this that almost surely, 
\begin{equation}\label{subcritic1} 
\limsup_{t\to \infty} \ \mathrm e^{-\lambda t} \cdot \Pi_t(x) = 0.
\end{equation}
Indeed, to see this, fix $x\in \mathbb Z^d$ and $\varepsilon>0$, and define inductively a sequence of stopping times $(\tau_k)_{k\ge 0}$ as follows. First $\tau_0= 0$, and then for $k\ge 0$, 
$$\tau_{k+1} = \inf \{t\ge \tau_k + 1 : \Pi_t(x) \ge \varepsilon \cdot \mathrm e^{\lambda t} \}. $$ 
Consider also for any $k\ge 0$, the event: 
$$A_k = \{\tau_k<\infty\}\cap \{\Pi_t(x) \ge \varepsilon \cdot \mathrm e^{\lambda t - 2}, \ \forall t\in [\tau_k,\tau_k+1]\}.$$ 
Since any particle alive at time $\tau_k$ at position $x$, will stay there until time $\tau_k +1$ with probability $1/e$, we have using standard large deviations results that for some constant $c>0$, for any $k\ge 0$, almost surely, 
$$\mathbb P(A_k\mid \mathcal F_{\tau_k}) {\bf 1}_{\{\tau_k<\infty\}} \ge c  {\bf 1}_{\{\tau_k<\infty\}}. $$ 
Thus, the conditional Borel-Cantelli lemma ensures that, almost surely on the event when $\tau_k<\infty$ for infinitely many $k$, one also has that 
$\Pi_n(x) \ge \varepsilon \mathrm e^{\lambda n - 2}$, for infinitely many $n$. 
In light of~\eqref{subcritic0}, we deduce that, almost surely, there exists $k$ such that $\tau_k = \infty$. 
Since this holds for any $\varepsilon>0$, this proves~\eqref{subcritic1}.  
Now note that the total number of particles is stochastically increasing in $\lambda_0$ and 
{if $\lambda_0= \lambda$, then $\mathrm e^{-\lambda t}N_t \to \xi$ almost surely as $t\to\infty$, 
where $\xi$ is a standard exponential random variable (see~\cite[Section~III.5]{AK}).}
This implies that, almost surely
\begin{equation}\label{subcritic2}
\liminf_{t\to \infty} \mathrm e^{-\lambda t}\cdot N_t >0.
\end{equation}
Combining~\eqref{subcritic1} and~\eqref{subcritic2} concludes the proof of of Theorem~\ref{prop:weakloc} (ii).

\begin{remark}\label{rem:subcritical}
In the proof of Theorem~\ref{prop:weakloc} (ii), we have used~\eqref{eq:ABY} to prove~\eqref{eq:tilde.u}.
Because the proof of ~\eqref{eq:ABY} is only sketched in~\cite{ABY98}, 
we provide here an alternate and simpler argument for~\eqref{eq:tilde.u}. 
First it is convenient to consider that in fact particles never die, and give birth to new particles at rate $\lambda_0-\lambda$ when sitting at the origin. 
In particular the times at which the initial particle gives birth to new particles forms a Poisson point process on $[0,\infty)$ with intensity given by 
\[\mathrm d\mu(t) = (\lambda_0-\lambda){\bf 1}_{X(t)= 0}\, \mathrm dt,\] 
where $X(t)$ is a simple symmetric random walk started at the origin and with jump rate~$1$.  Using Campbell's formula, this yields the recursive formula: 
\begin{equation}\label{rec.tildeu}
\widetilde u(x,t) 
= \mathbb E\bigg[\int_0^t \widetilde u(x,t-s)\,  \mathrm d\mu(s)\bigg]  + p_t(0,x) = (\lambda_0- \lambda) \int_0^t \widetilde u(x,t-s) p_s(0,0)\, \mathrm ds + p_t(0,x),
\end{equation}
where $p_s(0,x) = \mathbb P(X(s)= x)$. 
Integrating this equation from $t=0$ to infinity gives that, if $\lambda_0 <\lambda +\gamma_d$, then
\begin{equation}\label{int.tildeu}
\int_0^\infty \widetilde u(x,t) \, \mathrm dt =\frac{ G(x)}{1-m}<\infty,
\end{equation}
 where $G(x) = \int_0^\infty p_s(0,x)\, \mathrm ds $, and $m = (\lambda_0- \lambda)G(0)=\frac{\lambda_0- \lambda}{\gamma_d}<1$.

Let $C>0$ be some large constant to be fixed later. 
Define $h(x,t) = (1+t^{d/2}) \widetilde u(x,t)$, and consider $t_* = \inf \{s\ge 0 : h(x,s) > C\}$. Note that $h$ is a continuous function, and hence 
if $t_*$ is finite, then $h(x,t_*) = C$. We shall also use that by the local central limit theorem, one has for any $x\in \mathbb Z^d$, for some constant $c>0$, and all $s\ge 0$,
\begin{equation}\label{ps0x}
p_s(0,x) \le \frac{c}{1+s^{d/2}}. 
\end{equation} 
Assume that $t_*$ is finite. We then cut the integral appearing in~\eqref{rec.tildeu} into three pieces, corresponding to the time intervals $[0,\delta t_*]$, $[\delta t_*,(1-\delta)t_*]$ and $[(1-\delta)t_*,t_*]$, and bound the integral on each of them as follows.
First, by definition of $t_*$, we get 
$$\int_0^{\delta t_*} \widetilde u(x,t_*-s) p_s(0,0)\, \mathrm ds  
\le \frac{C}{1+[(1-\delta)t_*]^{d/2}} \cdot  \int_0^{\delta t_*} p_s(0,0)\, \mathrm ds 
\le \frac{CG(0)}{(1-\delta)^{d/2} (1+t_*^{d/2})}.$$ 
Next, we use~\eqref{ps0x} to obtain 
$$\int_{\delta t_*}^{(1-\delta)t_*} \widetilde u(x,t_*-s) p_s(0,0)\, \mathrm ds  \le  \frac{Cc t_*}{(1+(\delta t_*)^{d/2})^2} \le \frac{Cc t_*}{\delta^d (1+t_*^{d/2})^2},$$
and finally~\eqref{int.tildeu} yields 
$$\int_{(1-\delta)t_*}^{t_*} \widetilde u(x,t_*-s) p_s(0,0)\, \mathrm ds  \le \frac{c}{1+[(1-\delta)t_*]^{d/2}} \int_0^{\delta t_*}  \widetilde u(x,s)\, \mathrm ds
\le  \frac{cG(x)}{(1-m)(1-\delta)^{d/2}( 1+t_*^{d/2})}. $$ 
Altogether, and using again~\eqref{ps0x} for the last term in~\eqref{rec.tildeu}, we deduce that
$$h(x,t_*) \le  c + \frac {mC}{(1-\delta)^{d/2}} + \frac{c(\lambda_0-\lambda)G(x)}{(1-m)(1-\delta)^{d/2}} 
+ \frac{Cc(\lambda_0-\lambda)t_*}{\delta^d (1+ t_*^{d/2})}.$$
Now assume that $\delta$ is small enough, so that $\varepsilon := 1- \frac{m}{(1-\delta)^{d/2}}>0$. Then choose $C$ large enough, so that 
$$c+\frac{c(\lambda_0-\lambda)G(x)}{(1-m)(1-\delta)^{d/2}} \le \frac{\varepsilon}{4} C.$$ 
Let also $t_0$ be large enough, such that $\frac{c(\lambda_0- \lambda)t_0}{\delta^d(1+t_0^{d/2})} \le \varepsilon/4$ (note that this is possible since $d\ge 3$ by hypothesis), and take $C$ also larger than $\sup_{t\le t_0} h(x,t)$, so that $t_*\ge t_0$. With these choices of the constants, we obtain that $h(x,t_*)  \le (1- \frac{\varepsilon}{2})C$, contradicting the fact that $h(x,t_*)$ should be equal to $C$. Hence $t_* = \infty$, as wanted.

Note that using this Poisson representation, one can also infer that $\sup_{t\ge 0} \mathbb E[\widetilde N_t]<\infty$, when $\lambda_0<\lambda + \gamma_d$, which provides an alternate proof of the second part of~\eqref{eq:ABY}. Indeed, let $\widetilde N_\infty$ be the increasing limit of $\widetilde N_t$ as $t\to \infty$, which is also the total number of particles ever born in the process. 
By definition the total number of particles ever created by a single particle in its whole life {(we only consider direct offspring here, not all descendants)} is distributed as a Poisson random variable with mean $(\lambda_0-\lambda) \cdot L_\infty(0)$, where $L_\infty(0)$ is the total time it has spent at the origin. 
Thus, the mean number of offspring of each particle is equal to 
\[m=(\lambda_0-\lambda) \cdot \mathbb E[L_\infty(0)]=  \frac{\lambda_0-\lambda}{ \gamma_d}<1,\]
and consequently 
\[\sup_{t\ge 0} \mathbb E[\widetilde N_t] = \mathbb E[\widetilde N_\infty] = \sum_{n=0}^\infty m^n <\infty.\]
\end{remark}

\begin{remark}\label{rem:exp}
Thanks to Proposition~\ref{prop:exp} and the results proved in~\cite{ABY98}, we can infer some new first moment asymptotics in the critical and supercritical regimes, namely when $\lambda_0\ge  \lambda +\gamma_d$. More precisely, we can deduce from the results of~\cite{ABY98} that if $\lambda_0>\lambda+ \gamma_d$, there exist constants $C>0$, $\rho>\lambda$ and $(C(x))_{x\in \mathbb Z^d} \in (0,\infty)$, such that for all $x\in \mathbb Z^d$, 
\begin{equation*}
\lim_{t\to \infty} \mathrm e^{-\rho t}\cdot  \mathbb E[\Pi_t(x)]  = C(x), \quad \textrm{and}\quad 
\lim_{t\to \infty} \mathrm e^{-\rho t}\cdot \mathbb E[N_t]   = C.  
\end{equation*}
In the critical case $\lambda_0= \lambda + \gamma_d$, we can infer that $\mathrm e^{-\lambda t} \cdot \mathbb E[N_t]$ grows subexponentially fast. We refer to~\cite{ABY98} for more detailed results. 
\end{remark}


\section{Proof of Proposition~\ref{prop:balance} (``balance'' equations)}
\label{sec:balance}
We start by introducing a martingale which lies at the heart of the proof. 
\begin{lemma}\label{lem:balance1}
For all $t\ge 0$, we let $\rho_t = \lambda + (\lambda_0-\lambda)\widehat \Pi_t(0)$ and $M_t = N_t\cdot  \mathrm e^{-\int_0^t \rho_s\mathrm ds}$. 
Then the process $(M_t)_{t\ge 0}$ is a martingale bounded in~$L^2$. 
\end{lemma}
\begin{proof}
The fact that $N_t$ is square integrable for all $t\geq 0$ follows from the fact 
that $(N_t)_{t\geq 0}$ is stochastically dominated by the 
total population of a binary branching process with reproduction rate~$\lambda_0$, 
whose square is well-known to be integrable. 
For all $0\le s\le t$, we let $f_s(t) = \mathbb E[M_t|\mathcal F_s]$, 
where $(\mathcal F_s)_{s\ge 0}$ is the natural filtration of the process $(\Pi_s)_{s\ge 0}$. 
It is standard to see, using the definition of the process $(\Pi_t)_{t\geq 0}$ 
as a Markov jump process, 
that, for all $s\ge 0$, $t\mapsto f_s(t)$ is continuous on $[s,+\infty)$, and moreover, that for all $t>s$, $\frac{\mathrm d}{\mathrm dt} f_s(t)=0$, from which we deduce that $(M_t)_{t\ge 0}$ is indeed a martingale. 
To see that it is bounded in~$L^2$, note that for any fixed $t\ge 0$, 
\begin{align*}
\mathbb E[N_{t+h}^2 \mathrm e^{-2\int_0^t \rho_s\, ds}] 
& =  h\mathbb E[\rho_t M_t (N_t + 1)^2\mathrm e^{-\int_0^t\rho_s \, \mathrm ds}]  +\mathbb E[ (1 - h\rho_t N_t) M_t^2] + \mathcal O(h^2) \\ 
& = \mathbb E[M_t^2] + 2h\mathbb E[M_t^2 \rho_t] +  h \mathbb E[M_t\rho_t \mathrm e^{-\int_0^t \rho_s\, \mathrm ds}] + \mathcal O(h^2). 
\end{align*}
Hence, 
\begin{align*}
\mathbb E[M_{t+h}^2]
& =\mathbb E[N_{t+h}^2\cdot \mathrm e^{-2\int_0^t \rho_s\, \mathrm ds}] -2h \mathbb E[M_t^2 \rho_t] + \mathcal O(h^2) \\ 
& = \mathbb E[M_t^2] +   h \mathbb E[M_t\rho_t \mathrm e^{-\int_0^t \rho_s\, \mathrm ds}]+ \mathcal O(h^2).
\end{align*}
Using Cauchy-Schwarz inequality and the fact that, for all $t\geq 0$,
$\lambda \le \rho_t \le \lambda_0$, we get that 
\[\frac{\mathrm d}{\mathrm dt} \mathbb E[M_t^2] 
= \mathbb E[M_t\rho_t \mathrm e^{-\int_0^t \rho_s\, \mathrm ds}] \le \lambda_0 \mathbb E[M_t^2]^{1/2} \mathrm e^{-\lambda t}.\]
In other words,
\[\frac{\frac{\mathrm d}{\mathrm dt}\mathbb E[M_t^2]}{2\sqrt{\mathbb E[M_t^2]}}
\leq \frac{\lambda_0}2  \mathrm e^{-\lambda t}.\]
By integration, and because $\mathbb E[M_0^2] = 1$, we get
\[\sqrt{\mathbb E[M_t^2]}  \leq 1+\frac{\lambda_0}{2\lambda}.\]
Taking the square and then the supremum on $t$, we get
\[\sup_{t\ge 0}\mathbb E[M_t^2]  \le \Big( 1+ \frac{\lambda_0}{2\lambda}\Big)^2,\]
concluding the proof. 
\end{proof}

For the next step we need some additional notation: 
For all $x\in \mathbb Z^d$, we let $\Pi_t^x$ and $N_t^x$ be respectively the empirical measure of the particles and the total number of particles at time $t$ in the catalytic branching random walk starting with one particle at site $x$. 
Given any measure $\mu$ on $\mathbb Z^d$, we let 
$\Pi_t\cdot \mu = \sum_{x\in \mathbb Z^d} \mu(x)  \Pi_t^x$.  

\begin{lemma}\label{lem:balance2}
If there exists a probability measure~$\nu$ such that $\lim_{t\to \infty} \widehat \Pi_t = \nu$, almost surely for the topology of weak convergence,
then $\mathbb E[\Pi_t \cdot \nu]\cdot \mathrm e^{-\rho t} = \nu$, 
for all $t\ge 0$, with $\rho = \lambda + (\lambda_0- \lambda) \nu_0$.  
\end{lemma}
\begin{proof}
Fix $t\ge 0$. On one hand, we know by Lemma~\ref{lem:balance1} that $(M_t)_{t\ge 0}$ converges almost surely and in $L^2$ towards some random variable $M_\infty$. It follows that, for any $x\in \mathbb Z^d$, almost surely,
\[\Pi_{t+s}(x) \mathrm e^{-\int_0^{t+s} \rho_u\, \mathrm du}
= M_{t+s}\widehat \Pi_{t+s}(x) \ \underset{s\to \infty}{\longrightarrow} \ 
M_\infty \nu_x .\]  
Moreover, since, by definition, $|\widehat \Pi_{t+s}(x)|\le 1$, this convergence also holds in $L^1$, which implies
\begin{equation*}
   \lim_{s\to \infty} \mathbb E[\Pi_{t+s}(x) \mathrm e^{-\int_0^{t+s} \rho_u\, \mathrm du}] = \nu_x,
   \end{equation*}
   because $\mathbb E[M_\infty] = 1$.
Furthermore, because by assumption $\rho_u \to \rho$ almost surely as $u\uparrow\infty$, we get $\int_s^{t+s} \rho_u\, \mathrm du = \rho t + o(1)$ as $s\uparrow\infty$. 
Since for any fixed $t>0$, $\int_s^{t+s} \rho_u\, \mathrm du$ is uniformly bounded in $s$, and using that the martingale $(M_t)_{t\ge 0}$ is bounded in $L^2$ by Lemma~\ref{lem:balance1}, we deduce by dominated convergence, that
\begin{equation}\label{limitPinu2}
 \lim_{s\to \infty} \mathbb E[\Pi_{t+s}(x) \mathrm e^{-\rho t - \int_0^s \rho_u\, \mathrm du} ]= \nu_x.
 \end{equation}
Now note that, for all $s, t\ge 0$, 
if we let $(\widetilde \Pi_t)_{t\geq 0}$ be an independent copy of the process $(\Pi_t)_{t\geq 0}$, then $\Pi_{t+s} = \widetilde \Pi_t \circ \Pi_s$ in distribution, where, with a slight abuse of notation, we consider that $\widetilde \Pi_t$ acts as well on a particle configuration and on its empirical measure. 
By linearity of the expectation, this implies
\[\mathbb E[\Pi_{t+s}(x) \mathrm e^{- \int_0^s \rho_u\, \mathrm du} ] 
= \sum_{y\in \mathbb Z^d} \mathbb E[\Pi_t^y(x)]  \mathbb E[\Pi_s(y) \mathrm e^{- \int_0^s \rho_u\, \mathrm du}].\]
Applying the convergence result~\eqref{limitPinu2} with $t=0$, we get that, for all $y\in \mathbb Z^d$, 
\[\lim_{s\to \infty}  \mathbb E[\Pi_s(y) \mathrm e^{- \int_0^s \rho_u\, du}] = \nu_y.\]
It just remains now to invert the summation over $y\in \mathbb Z^d$ and the limit $s\uparrow \infty$, which we do using the dominated convergence theorem: 
First note that
\[\sup_{y\in \mathbb Z^d} \sup_{s\ge 0} \mathbb E[\Pi_s(y) \mathrm e^{- \int_0^s \rho_u\, \mathrm du}] 
\le \sup_{s\ge 0} \mathbb E[M_s] <\infty.\]
Also, by Cauchy-Schwarz's inequality, for all $y\in \mathbb Z^d$ and $t\ge 0$, 
\[\mathbb E[\Pi_t^y(x)] \le \mathbb E[(N_t^y)^2]^{\nicefrac12}  \mathbb P(\Pi_t^y(x)\neq 0)^{\nicefrac12}.\]
Moreover, when $\|y\|>2\|x\|$, we can bound the probability on the right-hand side using that the probability for a simple random starting from~$y$ to enter the ball $B(0,\|x\|)$ by time $t$ is bounded by $C \exp(- c\|y\|^2/t)$, for some 
positive constants $c$ and $C$. -Hence a union bound on the set of particles alive at time $t$ ensures that, for all $y\in\mathbb Z^d$ satisfying $\|y\|>2\|x\|$,
\[\mathbb P(\Pi_t^y(x)\neq 0) \le C \mathrm e^{\lambda t} \cdot \exp(- c\|y\|^2/t).\]
Finally, because the catalytic branching random walk is stochastically dominated by a branching random walk reproducing at constant rate $\lambda_0$ everywhere, 
we get that, for any fixed $t\ge 0$, 
\[\sup_{y\in \mathbb Z^d} \mathbb E[(N_t^y)^2] < \infty.\]
Thus, the dominated convergence theorem applies and gives that
\[\lim_{s\to \infty} \mathbb E[\Pi_{t+s}(x) \cdot \mathrm e^{- \int_0^s \rho_u\, \mathrm du} ] = \sum_{y\in \mathbb Z^d} \mathbb E[\Pi_t^y(x)] \cdot \nu_y.\]
Together with~\eqref{limitPinu2}, this concludes the proof. 
\end{proof}

We are now in a position to conclude the proof of our main result in this section:

\begin{proof}[Proof of Proposition~\ref{prop:balance}]
By Lemma~\ref{lem:balance2}, for all $x\in \mathbb Z^d$ and $t\ge 0$,
\[\nu_x = \sum_{y\in \mathbb Z^d} \nu_y \cdot \mathbb E[\Pi_t^y(x)] \cdot \mathrm e^{-\rho t}.\]
Taking the differential in $t$ at $t=0$, we get that, for all $x\neq 0$,  
\[0 = \nu_x (\lambda- 1-\rho) + \frac 1{2d} \sum_{y\sim x} \nu_y,\]
and when $x=0$, 
\[0 = \nu_0(\lambda_0 - 1 - \rho) + \frac 1{2d} \sum_{y\sim 0} \nu_y,\]
which after simplifying give respectively~\eqref{eq:balance1} and~\eqref{eq:balance2}, as desired.  
\end{proof}

\section{Proof of Theorem~\ref{th:no_loc}: existence of a ``stationary'' measure}
\label{sec:noloc}

\subsection{Proof of Theorem~\ref{th:no_loc}(i): non-localisation}
\label{sub:no_existence}
First note that any probability measure~$\nu$ satisfying~\eqref{eq:balance1} must satisfy $\nu_0>0$. Indeed, if $\nu_0=0$, then~\eqref{eq:balance1} and~\eqref{eq:balance2} (which follows from \eqref{eq:balance1} and the fact that $\sum_{x\in\mathbb Z^d} \nu_x = 1$) 
show that $\nu$ is harmonic. 
However, it is well-known that any 
bounded harmonic function on $\mathbb Z^d$ is constant, 
hence $\nu_x = 0$ for all $x\in\mathbb Z^d$, which is a contradiction. 

Now, let $\nu$ be a probability measure on $\mathbb Z^d$ satisfying~\eqref{eq:balance1}. 
To simplify notation, we set $\varepsilon = \lambda_0-\lambda$. 
Consider a simple random walk $(S_n)_{n\ge 0}$ on $\mathbb Z^d$, and let $\tau_0$ be the first return time to $0$, i.e.\ $\tau_0 = \inf\{n\ge 1 : S_n = 0\}$. 
Equation~\eqref{eq:balance1} shows that for any $x\neq 0$, 
on the event that $S_0 = x$, the process $(M_n)_{n\ge 0}$, defined for $n\ge 0$ by 
\[M_n = \frac{\nu_{S_{n\wedge \tau_0}}}{(1+\varepsilon \nu_0)^{n\wedge \tau_0}},\] 
is a bounded martingale. Hence, the optional stopping theorem gives
\begin{equation}\label{nux}
\nu_x  = \mathbb E_x\Big[\frac{\nu_0\cdot  \mathbf 1\{\tau_0<\infty\}}{(1+\varepsilon \nu_0)^{\tau_0}}\Big], 
\end{equation}
where $\mathbb E_x$ denotes expectation with respect to the law of the random walk starting from $x$.
Using now the symmetry of the walk, and letting $\tau_x$ denote the hitting time of $x$, we get 
\[\nu_x = \nu_0 \cdot \mathbb E\Big[  \frac{\mathbf 1\{\tau_x<\infty\}}{(1+\varepsilon \nu_0)^{\tau_x}}\Big] = \nu_0 \sum_{k=1}^\infty \frac{1}{(1+\varepsilon \nu_0)^k }\cdot \mathbb P(\tau_x= k).\]
Summing over $x\in \mathbb Z^d\setminus \{0\}$, and taking into account that $(\nu_x)_{x\in \mathbb Z^d}$ is a probability measure, we get 
\[1  = \nu_0 + \sum_{x\in \mathbb Z^d\setminus \{0\}} \nu_x  = \nu_0 + \nu_0\sum_{k = 1}^\infty \frac{1}{(1+\varepsilon \nu_0)^k }\cdot \mathbb P(S_k\notin \{S_0,\dots,S_{k-1}\} ).\]
Reversing time again, we deduce that 
\begin{align} \label{eq.3}
\nonumber 1 & = \nu_0 + \nu_0 \sum_{k = 1}^\infty \frac{1}{(1+\varepsilon \nu_0)^k }\cdot \mathbb P(0\notin \{S_1,\dots,S_k\} )\\
& = \nu_0 + \nu_0 \sum_{k = 1}^\infty \frac{1}{(1+\varepsilon \nu_0)^k }\cdot \mathbb P(\tau_0 \ge k+1). 
\end{align}
Note now that for any $k\ge 1$, one has $\mathbb P(\tau_0\ge k+1)\ge \mathbb P(\tau_0= \infty)=\gamma_d$. Thus~\eqref{eq.3} gives
\[1 \ge  \nu_0 + \nu_0 \gamma_d \sum_{k = 1}^\infty \frac{1}{(1+\varepsilon \nu_0)^k } = \nu_0 + \frac{\gamma_d}{\varepsilon},\]
which yields a contradiction if $\varepsilon\le \gamma_d$, since we recall that $\nu_0>0$.

\subsection{Proof of Theorem~\ref{th:no_loc}(ii): existence of a stationary measure}
As proved in Section~\ref{sub:no_existence}, any probability measures~$\nu$ satisfying~\eqref{eq:balance1} satisfies $\nu_0>0$.
For any positive real number $\nu_0$, we use~\eqref{nux} 
to define a positive measure $\nu$ on $\mathbb Z^d$ that satisfies~\eqref{eq:balance1}, by the Markov property. 
Now note that a simple random walk on $\mathbb Z^d$ started at some $x$ cannot reach~$0$ before time $\|x\|_{\infty}$.
Thus, $\nu_x \le \nu_0(1+\varepsilon \nu_0)^{-\|x\|_\infty}$, which is summable. 
Hence it only remains to show that, for all $\varepsilon>\gamma_d$, 
there exists a unique value of $\nu_0$ making $\nu$ a probability measure. 
This is equivalent to showing that there exists a unique value of $\nu_0$ such that~\eqref{eq.3} is satisfied. 
Letting $u=\varepsilon \nu_0$, we see that this is equivalent to the existence and uniqueness of a solution to the equation $f(u) = \varepsilon$, 
where $\varepsilon>\gamma_d$ is fixed and for $u>0$, 
\begin{equation}\label{eq.u}
f(u) = u + u \sum_{k=1}^\infty \frac {1}{(1+u)^k} \cdot \mathbb P(\tau_0\ge k+1). 
\end{equation}
Noting that $\mathbb P(\tau_0\ge k+1) =\gamma_d +  \sum_{i\ge k+1} \mathbb P(\tau_0 = i)$, and inverting the order of summation, we get 
$$f(u) = u + \gamma_d + \sum_{i=2}^\infty  \mathbb P(\tau_0=i)\cdot \Big(1- \frac{1}{(1+u)^{i-1}}\Big). $$ 
Noting also that $\tau_0\ge 2$ almost surely, this gives 
$$f(u) = 1+ u - \sum_{i=2}^\infty  \frac{\mathbb P(\tau_0=i)}{(1+u)^{i-1}}
=1+  u - \mathbb E\Big[\frac{1}{(1+u)^{\tau_0-1}}\Big]. $$ 
(Note that this last expression of $f$ could also have been derived directly from~\eqref{eq:balance2} and~\eqref{nux}.)
In particular $f$ is strictly increasing on $(0,\infty)$, converges to $+\infty$, as $u$ goes to infinity, and by dominated convergence, we can see that it converges to $\gamma_d$ as $u$ goes to zero,
which proves well that for any $\varepsilon>\gamma_d$, the equation $f(u)=\varepsilon$ has a unique solution.

\begin{remark}
Note that the last expression of $f$ above rewrites as 
$$f(u ) =(1+ u)\cdot \Big(1  - G\Big(\frac 1{1+u}\Big)\Big),$$
where $G(z) = \mathbb E[z^{\tau_0}]$, is the generating function of $\tau_0$. Now for $n\ge 0$, denote by 
$\tau_n$ the $(n+1)$-th return time to the origin of the walk, and let $G_n(z) = \mathbb E[z^{\tau_n}]$. 
By independence between the successive return times to the origin, one has 
$G_n(z) = G(z)^{n+1}$, for any $n\ge 0$. In particular one has for any $0<z<1$, 
\begin{equation}\label{taun}
 \sum_{n\ge 0} G_n(z) = \frac{G(z)}{1-G(z)}.
\end{equation}
Now in~\cite{DR}, branching occurs only at the origin and at rate $1$, which corresponds to choosing $\lambda = 0$ and $\lambda_0 = 1$ in our notation. In this case, when furthermore each particle splits into two particles at each branching event (as in our case), they show that the exponential growth rate $\rho$ of the catalytic branching random walk (the so-called Malthusian parameter) is solution of the equation 
\begin{equation} \label{conditionDR}
\int_0^\infty \mathrm e^{-\rho t}\,  p_t(0,0)\, dt = 1,
\end{equation}
where $p_t(0,0)$ is the probability that a continuous time simple random walk is at the origin at time $t$. Since by definition the time between any two consecutive jumps is a mean one exponential random variable, one also has, with $\widetilde \tau_n$ the $n$-th return time to the origin for the continuous time random walk,  
\begin{align*}
 \int_0^\infty \mathrm e^{-\rho t}\,  p_t(0,0)\, dt 
 & = \frac 1{1+\rho} + \sum_{n\ge 1} \int_0^\infty \mathrm e^{-\rho t}\cdot \mathbb E\big[\mathrm e^{-(t-\widetilde \tau_n)} \cdot \mathbf 1\{\widetilde \tau_n \le t\} \big] \, dt \\
 & = \frac 1{1+\rho} + \frac 1{1+\rho}  \sum_{n\ge 1} \mathbb E[\mathrm e^{-\rho \widetilde \tau_n}] \\ 
 & = \frac 1{1+\rho} + \frac 1{1+\rho}\sum_{n\ge 0} G_n\Big(\frac 1{1+\rho}\Big).
 \end{align*} 
Hence, using~\eqref{taun}, we see that~\eqref{conditionDR} is equivalent to 
$$G\Big(\frac 1{1+\rho}\Big) = \frac{\rho}{1+\rho}.$$
In particular we recover well the same equation for $\rho$ as our equation defining $\nu_0$, namely $f(\nu_0) = 1$, which agrees with the fact that in our setting, when $\lambda_0=1$, and in the limit $\lambda\to 0$, the exponential growth rate is given by $\nu_0$, as shown by combining Lemma~\ref{lem:balance1} and~Theorem~\ref{th:main}. 
\end{remark}

\section{Proof of localisation (proof of Theorem~\ref{th:main})}
\label{sec:mvpp}
In this section, we assume that $\lambda_0>2d-1+2d\lambda$.
For the proof of Theorem~\ref{th:main}, we use the fact that $(\Pi_t)_{t\geq 0}$, taken at the times when it changes values is a ``measure-valued P\'olya process'' (or, in other words, an infinitely-many-colour P\'olya urn); we thus start the section with some useful background and existing results on these processes.

\subsection{Measure-valued P\'olya processes}

\begin{definition}
Let $E$ be a Polish space, $\pi$ a finite measure on $E$, 
$R^{\sss (1)} = (R^{\sss (1)}_x)_{x\in E}$ be a random kernel on $E$ 
(i.e., for all $x\in E$, $R^{\sss (1)}_x$ is a random measure on~$E$, almost surely finite),
and $P = (P_x)_{x\in E}$ is a kernel on $E$ (i.e., for all $x\in E$, $P_x$ is a finite measure on~$E$).
The measure-valued P\'olya process (MVPP) of initial composition $\pi$, replacement kernel $R^{\sss (1)}$ and weight kernel $P$ is the sequence of random measures $(m_n)_{n\geq 0}$ defined recursively as follows:
$m_0 = \pi$ and, for all $n\geq 0$,
\[m_{n+1} = m_n + R^{(n+1)}_{\xi(n+1)},\]
where, given $m_n$, $\xi(n+1)$ is a random variable of distribution $m_n P/m_nP(E)$
with 
\[m_nP = \int_{x\in E} P_x\mathrm dm_n(x),\]
and where, given $\xi(n+1)$, $R^{\sss (n+1)}_{\xi(n+1)}$ is an independent copy of $R^{\sss (1)}_{\xi(n+1)}$.
\end{definition}

\begin{lemma}
For all $n\geq 0$, let $\tau_n$ be the time of the $n$-th event (jump of a particle or birth of a particle).
Also let $m_n = \frac1\kappa \Pi_{\tau_n}$ for all $n\geq 0$, where 
\begin{equation}\label{eq:def_kappa}
\kappa = \lambda_0 - \frac{\lambda_0-\lambda}{1+\lambda_0}.
\end{equation}
Then, $(m_n)_{n\geq 0}$ is an MVPP with the following parameters:
\begin{itemize} 
\item initial composition $m_0 = \frac1\kappa\delta_{\bf 0}$; 
\item replacement kernel $(R_x^{\sss (1)})_{x\in\mathbb Z^d}$ where, for all $x\in\mathbb Z^d$,
\[R^{\sss (1)}_x = \frac1{\kappa}\big(B_x\delta_x + (1-B_x)(\delta_{x+\Delta}-\delta_x)\big)
= \frac1{\kappa}\big((2B_x -1)\delta_x + (1-B_x)\delta_{x+\Delta}\big),\]
where $B_x$ is a Bernoulli random variable of parameter $\frac{\lambda_x}{1+\lambda_x}$, 
with $\lambda_x = \lambda_0$ if $x = 0$ and $\lambda_x = \lambda$ otherwise,
and where $\Delta$ is a simple symmetric random walk increment independent of $B_x$;
\item weight kernel $((1 + \lambda_x)\delta_x)_{x\in \mathbb Z^d}$.
\end{itemize}
\end{lemma}

\begin{remark}
The reason why we divide $\Pi_{\tau_n}$ by $\kappa$ in the definition of $m_n$ is technical, and will be discussed later (we need that $\sup_{x} \mathbb E[R^{\sss (1)}_x(\mathbb Z^d)]\leq 1$). Note, however, that $\kappa>0$.
\end{remark}

To prove localisation, 
we use~\cite[Theorem~1]{MV20} 
(see Section~1.4 where the case of $R^{\sss (1)}$ being a signed-measure is discussed). 
We prove that $(m_n)_{n\geq 0}$ satisfies the assumptions (T), (A'1), (A'2), (A'3) and (A4) below, which are a slight modification of the assumptions in~\cite{MV20} (we replace (A1) and (A3) therein by two other assumptions called (A'1) and (A'3)). 
First define $R = \mathbb E[R^{\sss (1)}]$, meaning that, for all $x\in E$, for all measurable sets $A\subset E$, \[R_x(A) = \mathbb E[R^{\sss (1)}_x(A)].\] 
We also define, for all $x\in E$,
\begin{equation}\label{eq:defQ}
Q^{\sss (1)}_x = \sum_{y\in E} R^{\sss (1)}_x(\{y\}) P_y 
\quad\text{ and }\quad 
Q_x = \sum_{y\in E}  R_x(\{y\}) P_y,
\end{equation}
or, in other words, $Q^{\sss (1)} = R^{\sss (1)}P$ and $Q = RP$.
 \begin{itemize}
 \item[(T)] For all $n\geq 0$, $m_n$ is a positive measure. 
\item[(A'1)] For all $x\in E$, $Q_x(E)\leq 1$.
Furthermore, there exists $c>0$ such that
\[\liminf_{n\uparrow\infty} \frac{m_nP(E)}{n}
= \liminf_{n\uparrow\infty} \frac1n \sum_{i=1}^n Q^{\sss (i)}_{\xi(i)}(E)
\geq c
\quad\text{ and }\quad
\liminf_{n\uparrow\infty} \frac1n \sum_{i=1}^n Q_{\xi(i)}(E)\geq c.\]
%
\item[(A'2)] there exist a locally bounded function
$V\,:\,E\to [1,+\infty)$ and some constants $r>1$, $p> 2$,
$q'>q:=p/(p-1)$, $\theta\in(0,c)$, $K>0$, $A\geq 1$, and $B\geq 1$, such that
\begin{itemize}
\item[(i)] for all $N\geq 0$, 
the set $\{x\in E \colon V(x)\leq N\}$ is relatively compact.
\item[(ii)]  for all $x\in E$,
\[Q_x \cdot V\leq \theta V(x) + K \quad\text{and}\quad Q_x \cdot V^{\nicefrac1q}\leq \theta V^{\nicefrac1q}(x) + K\]
\item[(iii)] for all continuous functions $f: E\rightarrow\mathbb R$ bounded by~$1$ and all $x\in E$, 
  \[|Q_x\cdot f|^{q'} 
  \vee \mathbb E \big[\big| {R}^{\sss (1)}_{x} \cdot f - R_{x}\cdot f\big|^r\big]
  \vee \mathbb E\big[\big| {Q}^{\sss (1)}_{x} \cdot f - Q_{x}\cdot f\big|^p\big] 
  \leq AV(x),\]
\item[(iv)] and
  \[|Q_x\cdot V^{\nicefrac1q}|^{q}\vee |Q_x\cdot V|\vee 
\mathbb E\left[\left|Q^{\sss (1)}_x\cdot V^{\nicefrac1q}-Q_x\cdot V^{\nicefrac1q}\right|^{r}\right] \leq B V(x).\]
\end{itemize}
\item[(A'3)] the continuous-time pure jump Markov process~$X$ with
  sub-Markovian jump kernel $Q-I$ admits a quasi-stationary
    distribution $\nu\in\mathcal P(E)$ (the set of all probability measures on~$E$).
    Also, the
  convergence of $\mathbb P_{\alpha}(X_t\in \cdot \,|\, X_t \neq \partial)$
  holds uniformly with respect to the total variation norm on a neighbourhood of $\nu$ in
  $\mathcal P_C(E)=\{\alpha\in\mathcal P(E)\mid \alpha \cdot V^{1/q}\leq C\}$, for each
  $C>0$, where $q = p/(p-1)$.
  Finally, for all $\alpha\in\mathcal P_C(E)$, $\mathbb P_\alpha(X_t\neq\partial)\geq \mathrm e^{-(1-c_0)t}$ for some constant $c_0>\theta$.
\item[(A4)] for all bounded continuous functions $f:E\rightarrow\mathbb R$,
  $x\in E \mapsto R_x\cdot f$ and $x\in E\mapsto Q_x\cdot f$ are continuous.
\end{itemize}

\begin{remark} By definition, i.e.\ because $m_n$ is the occupation measure of the catalytic random walk at the $n$-th jumping or branching event, (T) holds.
\end{remark}

\begin{remark} We have replaced Assumption (A1) in~\cite{MV20} by Assumption (A'1).
The fact that this can be done can be checked by going through the proofs in~\cite{MV20}, which adapt straightforwardly.
Indeed, here is a list of the changes one needs to make:
\begin{itemize}
\item In \cite[Lemma~3]{MV20}, one needs to change the definition of $\sigma_k$ to
\[\sigma_k = \inf\{n\geq k \colon m_nP(E)<c'n \text{ or }\sum_{i=1}^n Q_{\xi(i)}(E)<c'n\},\]
where $c'\in(\theta, c)$. The proof of Lemma 3 then directly follows from (A'1).
\item Assumption (A1) is then used three times in the proof of \cite[Lemma~5]{MV20}; the same three arguments go through by using the fact that on $\{\ell\leq\sigma_k\}$, $\tilde \eta_{\ell-1}Q(E):= \sum_{i=1}^{\ell-1} Q_{\xi(i)}(E)\geq c'(\ell-1)$ (instead of $\tilde \eta_{\ell-1}Q(E)\geq c_1 = c$, where $c_1$ is the notation used in~\cite{MV20} for our constant $c$).
\item Assumption (A1) is also used in the proof of \cite[Lemma~6]{MV20}.
There, one has to first fix an integer $k\geq 0$ and prove that there exists a random value of $C = C(k)$ such that $\tilde\eta_n \in \mathcal P_C(E)$ (the set of all probability measures $\mu$ on $E$ such that $\mu\cdot V^{\nicefrac1q}\leq C$) for all $n\leq \sigma_k$.
Instead of using Assumption (A1), we use the fact that $\tilde\eta_n(E)\leq c'n$ for all $n\leq \sigma_k$; with this modification, the proof goes through replacing $c_1$ ($=c$ in this paper) 
by $c'$ and only considering~$n$ up to~$\sigma_k$.
Because, by~\cite[Lemma~3]{MV20}, $\mathbb P(\cup_{k\geq 0}\{\sigma_k = \infty\}) = 1$,  
this is indeed enough to imply the conclusion of~\cite[Lemma~6]{MV20}.
\item Assumption (A1) is also used in the proof of \cite[Lemma~7]{MV20}: the authors use it to prove that, for any $\alpha\in\mathcal P_C(E)$, $\mathbb P_{\alpha}(X(t)\in E)\geq \mathrm e^{-(1-c_0)t}$ for some $c_0>\theta$ (in \cite[Lemma~7]{MV20}, $c_0=c$ but this is not necessary). 
This is now given by Assumption (A'3) and thus the proof of \cite[Lemma~7]{MV20} goes through under our modified set of assumptions: one only needs to replace every occurrence of $c_1$ by $c_0$.
\item Finally, Assumption (A1) is used in the proof of Proposition~7 (Step~1 of the proof): 
the authors use the fact that $m_nP(E)\geq c_1 n$ ($=cn$ in this paper) almost surely for all $n\geq 1$.
Instead, we can use the fact that $m_nP(E)\geq cn/2$ almost surely for all $n$ large enough.
\end{itemize}
\end{remark}

\begin{remark}
In Assumption (A3) in~\cite{MV20}, it is assumed that the convergence of 
$\mathbb P_{\alpha}(X_t\in \cdot \,|\, X_t \neq \partial)$
  holds uniformly with respect to the total variation norm in
  $\{\alpha\in\mathcal P(\mathbb Z^d)\mid \alpha \cdot V^{1/q}\leq C\}$, for each
  $C>0$.
  We are not able to prove that this assumption holds for our model, and hence replace (A3) by the weaker (A'3).
\end{remark}

\begin{remark} Because $\mathbb Z^d$ is discrete, (A4) holds straightforwardly.
\end{remark}

The following theorem is a close adaptation (the only difference being the assumption (A'3)) of Theorem~1 of~\cite{MV20}:
\begin{theorem}
\label{th:MV20}
Under Assumptions {\rm (T)}, {\rm (A'1)}, {\rm (A'2)}, {\rm (A'3)}, and {\rm (A4)}, 
if $m_0\cdot V<\infty$ and $m_0P\cdot V<\infty$, then the sequence of random measures
$(m_n/n)_{n\geq 0}$ converges almost surely to $\nu R$ with respect to the topology of weak convergence.  
Furthermore, 
if $\nu R(E)>0$, then $(\tilde{m}_n := m_n/m_n(E))_{n\in\mathbb N}$
converges almost surely to $\nu R/\nu R(E)$ with respect to the
topology of weak convergence.
\end{theorem}

\begin{proof}
We only need to explain how to adapt the proof of~\cite{MV20} to the case when
(A'3) holds instead of (A3). To do so, we need to introduce some notation from~\cite{MV20}:
for all $n\geq 0$, we let 
\[\eta_n = \sum_{i=1}^n \delta_{\xi(i)}.\]
We also let $\tilde \eta_0 = 0$ and, for all $n\geq 1$,
$\tilde \eta_n = \frac{\eta_n}{n}$.
The proof of~\cite{MV20} relies on the fact that $(\tilde\eta_n)_{n\geq 0}$ is a stochastic approximation: indeed, for all $n\geq 0$, one can write
\[\tilde \eta_{n+1} = \tilde \eta_n + \gamma_{n+1}\big(F(\tilde\eta_n) + U_{n+1}\big),\]
where $F(\mu) = \mu Q - \mu Q(E)\mu$, $U_{n+1} = \tilde\eta_n Q(E)\delta_{\xi(n+1)}-\tilde\eta_nQ$, and $\gamma_n = \frac1{n\tilde \eta_{n-1} Q(E)}$.

We now introduce some continuous-time interpolation of $(\tilde\eta_n)_{n\geq 0}$:
For all $n\geq 0$, we set $\tau_n = \gamma_1 + \cdots + \gamma_n$ and, for all $n\geq 0$, for all $t\in [\tau_n, \tau_{n+1}]$,
\[\tilde\mu_t = \tilde\eta_n + \frac{t-\tau_n}{\tau_{n+1}-\tau_n} (\tilde\eta_{n+1}-\tilde\eta_n).\]
Remark 12 in~\cite{MV20} states that, without Assumption (A3), one can still conclude that $(\tilde\mu_t)_{t\geq 0}$
is an asymptotic pseudo-trajectory in $\mathcal P_C(E)$ for the semi-flow induced in $\mathcal P_C(E)$ by the well-posed dynamical system 
\begin{equation}\label{eq:flow}
\frac{\mathrm d\mu_t \cdot f}{\mathrm dt} = \mu_t Q \cdot f - \mu_t Q(E)\mu_t \cdot f,
\end{equation}
(for all bounded functions $f = E\mapsto \mathbb R$).
Now, by Theorem~3.7 in~\cite{Benaim99}, 
the limit set of a pre-compact asymptotic pseudo-trajectory 
is internally chain transitive and thus (by~\cite[Proposition~5.3]{Benaim99}) 
it does not contain any proper attractor.
Now let $(X_t)_{t\geq 0}$ be the continuous-time pure jump Markov process with
sub-Markovian jump kernel $Q-I$; if, for all $t\geq 0$, we let $\mu_t$ be the distribution of $X_t$ conditioned on having survived until time $t$, then $(\mu_t)_{t\geq 0}$ is a solution of~\eqref{eq:flow}.
By Assumption (A'3), we get that $\nu$ is an attractor for the flow defined by~\eqref{eq:flow} on $\mathcal P_C(E)$. Also, all trajectories started in $\mathcal P_C(E)$ converge to~$\nu$. Thus, the limit set of $(\tilde\eta_t)_{t\geq 0}$ is $\{\nu\}$. 
Therefore, $\tilde\eta_n \to \nu$ almost surely as $n\uparrow\infty$, for the topology of weak convergence.
One can then follow the rest of the arguments in~\cite{MV20} to conclude the proof.
\end{proof}

\subsection{Checking Assumption (A'1)}
First note that, by definition (see Equation~\eqref{eq:defQ}), for all $x\in \mathbb Z^d$,
\[\kappa Q^{\sss (1)}_x = (2B_x -1)(1+\lambda_x)\delta_x + (1-B_x)(1+\lambda_{x+\Delta})\delta_{x+\Delta},\]
and thus
\ban
\kappa Q_x 
&= \Big(\frac{2\lambda_x}{1+\lambda_x}-1\Big)(1+\lambda_x)\delta_x + \frac{1}{1+\lambda_x}\mathbb E[(1+\lambda_{x+\Delta})\delta_{x+\Delta}]\notag\\
&= (\lambda_x-1)\delta_x + \frac{1}{1+\lambda_x}\mathbb E[(1+\lambda_{x+\Delta})\delta_{x+\Delta}].
\label{eq:Qx}
\ean
If $x = 0$, then $x+\Delta \neq 0$ almost surely and thus $\lambda_{x+\Delta} = \lambda$ almost surely, and hence,
\[\kappa Q_0(\mathbb Z^d) = \lambda_0-1 + \frac{1+\lambda}{1+\lambda_0}
= \lambda_0 - \frac{\lambda_0-\lambda}{1+\lambda_0}.\]
If $\|x\|_1 > 1$, then $x+\Delta \neq 0$ almost surely, and hence,
\[\kappa Q_x(\mathbb Z^d) = \lambda-1 + 1 = \lambda.\]
Finally, if $\|x\|_1 = 1$, then $x+\Delta = 0$ with probability $1/(2d)$, and hence, 
\[\kappa Q_x(\mathbb Z^d) = \lambda-1 +  \frac{1}{1+\lambda}\Big(\frac{2d-1}{2d} (1+\lambda) + \frac1{2d}(1+\lambda_0)\Big)
= \lambda + \frac{1}{2d\cdot (1+\lambda)}(\lambda_0-\lambda).\]
In summary,
\begin{equation}\label{eq:QE}
\kappa Q_x(\mathbb Z^d)
= \begin{cases}
\lambda_0 - \frac{1}{1+\lambda_0}(\lambda_0-\lambda) & \text{ if }x=0\\
\lambda + \frac{1}{{2d\cdot}(1+\lambda)}(\lambda_0-\lambda) & \text{ if }\|x\|_1=1\\
\lambda &\text{ if }\|x\|_1>1.
\end{cases}
\end{equation}
Note that, under the assumption that $\lambda_0>\lambda$ and $\lambda_0>2d-1+2d\lambda$, we have
\[\lambda_0 - \frac{\lambda_0-\lambda}{1+\lambda_0}
= \lambda + (\lambda_0-\lambda)\cdot \frac{\lambda_0}{1+\lambda_0}
> \lambda  + (\lambda_0-\lambda)\cdot \frac{2d-1+2d\lambda}{2d(1+\lambda)}
\geq \lambda + \frac{\lambda_0-\lambda}{2d(1+\lambda)},\]
because $2d-1+2d\lambda\geq 2d-1\geq 1$.
Thus the maximum of $x\mapsto \kappa Q_x(\mathbb Z^d)$ 
is reached at $x = 0$ and because, by definition, $\kappa = Q_0(\mathbb Z^d)$ (see~\eqref{eq:def_kappa}; this is why we chose this value for $\kappa$), we get
\begin{equation}\label{eq:maxQ}
\max_{x\in\mathbb Z^d} Q_x(\mathbb Z^d) 
= 1,
\end{equation}
as needed.
{Now note that, by definition, 
\[m_{n+1}P(\mathbb Z^d) = m_nP(\mathbb Z^d) + Q_{Y_{n+1}}^{\sss (n+1)}(\mathbb Z^d),\]
where, for all $x\in\mathbb Z^d$,
\[\mathbb P(Y_{n+1} = x|m_n) = \frac{m_nP(x)}{m_nP(\mathbb Z^d)}.\]
This implies
\begin{align}
m_nP(\mathbb Z^d) 
&= (1+\lambda_0) + \sum_{i=1}^n Q_{\xi(i)}^{\sss (i)}(\mathbb Z^d)\notag\\
&= (1+\lambda_0) + \sum_{i=1}^n \mathbb E_{i-1}[Q_{\xi(i)}^{\sss (i)}(\mathbb Z^d)] 
+ \sum_{i=1}^n \big(Q_{\xi(i)}^{\sss (i)}(\mathbb Z^d) - \mathbb E_{i-1}[Q^{\sss (i)}_{\xi(i)}(\mathbb Z^d)]\big),\label{eq:gnagna}
\end{align}
where $\mathbb E_{i-1}$ denotes the conditional expectation with respect to $m_1, \ldots, m_{i-1}$, for all $i\geq 1$. 
We treat the two sums above separately.
First note that, for all $i\geq 1$, by independence of $Q^{\sss (i)}$ and $m_1, \ldots, m_{i-1}$,
\begin{equation}\label{eq:repeat1}
\mathbb E_{i-1}[\kappa Q^{\sss (i)}_{\xi(i)}(\mathbb Z^d)] 
= \sum_{x\in\mathbb Z^d} \kappa Q_x(\mathbb Z^d)\cdot \frac{m_nP(x)}{m_nP(\mathbb Z^d)}
\geq \lambda + (\kappa-\lambda) \frac{m_nP(0)}{m_nP(\mathbb Z^d)}.
\end{equation}
Indeed, $\kappa Q_x(\mathbb Z^d)\geq \lambda$ for all $x\in \mathbb Z^d$ and $\kappa Q_0(\mathbb Z^d) = \kappa$.
By definition, $m_nP(0) = (1+\lambda_0)\Pi_{\tau_n}(0)$ and $m_nP(\mathbb Z^d) = (1+\lambda)N_{\tau_n} + (\lambda_0-\lambda)\Pi_{\tau_n}(0)\leq (1+\lambda_0)N_{\tau_n}$. Thus,
\[\liminf_{n\to\infty}\frac{m_nP(0)}{m_nP(\mathbb Z^d)}
\geq \liminf_{n\to\infty}\frac{\Pi_{\tau_n}(0)}{N_{\tau_n}}\geq \frac{\lambda_0-\lambda-1}{\lambda_0-\lambda},\]
almost surely,
by Theorem~\ref{prop:weakloc} (i) (because $\tau_n\to+\infty$ almost surely as $n\to+\infty$). 
Thus, almost surely,
\[\liminf_{n\to\infty} \mathbb E_{i-1}[\kappa Q^{\sss (i)}_{\xi(i)}(\mathbb Z^d)] 
\geq \lambda + (\kappa-\lambda) \frac{\lambda_0-\lambda-1}{\lambda_0-\lambda},
\]
which implies
\begin{equation}\label{eq:first_sum}
\liminf_{n\to\infty}\frac1n\sum_{i=1}^n \mathbb E_{i-1}[\kappa Q_{\xi(i)}^{\sss (i)}(\mathbb Z^d)]
\geq  \lambda + (\kappa-\lambda) \frac{\lambda_0-\lambda-1}{\lambda_0-\lambda}.
\end{equation}
We now control the second sum in~\eqref{eq:gnagna}: 
note that $(\Phi_n := \sum_{i=1}^n (Q_{\xi(i)}^{\sss (i)}(\mathbb Z^d) - \mathbb E_{i-1}[Q^{\sss (i)}_{\xi(i)}(\mathbb Z^d)]))_{n\geq 0}$ is a martingale.
Also,
\[\sup_{x\in \mathbb Z^d}\, \mathbb E\big|Q_x^{\sss (1)}(\mathbb Z^d)-Q_x(\mathbb Z^d)\big|^2<+\infty.\]
Indeed, for all $x\in\mathbb Z^d$, $Q_x^{\sss (1)}(\mathbb Z^d)$ is almost surely bounded (by~$-(1+\lambda_0)/\kappa$ from below, and by $2(1+\lambda_0)/\kappa$ from above).
This implies that
\[\langle \Phi\rangle_n 
= \sum_{i=1}^n \mathbb E_{i-1}\big[(Q_{\xi(i)}^{\sss (i)}(\mathbb Z^d) - \mathbb E_{i-1}[Q^{\sss (i)}_{\xi(i)}(\mathbb Z^d)])^2\big]
\leq n\sup_{x\in\mathbb Z^d} \mathbb E\big|Q_x^{\sss (1)}(\mathbb Z^d)-Q_x(\mathbb Z^d)\big|^2
=\mathcal O(n),\]
almost surely as $n\to+\infty$.
Thus, by the law of large numbers for martingales (see, e.g.\ \cite[Theorem~1.3.15]{Duflo}), $\lim_{n\to\infty}\Phi_n/n = 0$ almost surely.
Using this and~\eqref{eq:first_sum} in~\eqref{eq:gnagna} gives
\begin{equation}\label{eq:first_fix}
\liminf_{n\to\infty}\frac{m_nP(\mathbb Z^d)}{n}\geq \frac{\lambda}{\kappa} + \bigg(1-\frac\lambda\kappa\bigg) \frac{\lambda_0-\lambda-1}{\lambda_0-\lambda}.
\end{equation}
For $\liminf_{n\uparrow\infty} \frac1n \sum_{i=1}^n Q_{\xi(i)}(\mathbb Z^d)$, we reason similarly:
first note that, for all $n\geq 1$,
\begin{equation}\label{eq:two_sums2}
\frac1n \sum_{i=1}^n Q_{\xi(i)}(\mathbb Z^d)
= \frac1n \sum_{i=1}^n \mathbb E_{i-1}[Q_{\xi(i)}(\mathbb Z^d)]
+ \frac1n \sum_{i=1}^n \big(Q_{\xi(i)}(\mathbb Z^d)-\mathbb E_{i-1}[Q_{\xi(i)}(\mathbb Z^d)]\big).
\end{equation}
We look at these two sums separately: for the first sum, note that
\[\frac1n \sum_{i=1}^n \mathbb E_{i-1}[\kappa Q_{\xi(i)}(\mathbb Z^d)]
= \frac1n \sum_{i=1}^n \sum_{x\in\mathbb Z^d} \kappa Q_x(\mathbb Z^d) \frac{m_nP(x)}{m_nP(\mathbb Z^d)} 
= \frac1n \sum_{i=1}^n \mathbb E_{i-1}[\kappa Q^{\sss (i)}_{\xi(i)}(\mathbb Z^d)],\]
by~\eqref{eq:repeat1}. Thus, by~\eqref{eq:first_sum}, we get
\begin{equation}\label{eq:first_sum2}
\liminf_{n\to\infty}\frac1n \sum_{i=1}^n \mathbb E_{i-1}[\kappa Q_{\xi(i)}(\mathbb Z^d)]
\geq  \lambda + (\kappa-\lambda) \frac{\lambda_0-\lambda-1}{\lambda_0-\lambda}.
\end{equation}
For the second sum in~\eqref{eq:two_sums2}, we note that $(\hat\Phi_n := \sum_{i=1}^n (Q_{\xi(i)}(\mathbb Z^d)-\mathbb E_{i-1}[Q_{\xi(i)}(\mathbb Z^d)]))_{n\geq 0}$ is a martingale such that
\[\langle \hat\Phi\rangle_n
= \sum_{i=1}^n \mathbb E_{i-1}\big[\big(Q_{\xi(i)}(\mathbb Z^d)-\mathbb E_{i-1}[Q_{\xi(i)}(\mathbb Z^d)]\big)^2\big]
\leq 4 n\sup_{x\in\mathbb Z^d} Q_x(\mathbb Z^d)^2 = \mathcal O(n).\]
Thus, almost surely as $n\uparrow\infty$, $\hat\Phi_n = o(n)$. Thus, by~\eqref{eq:two_sums2} and~\eqref{eq:first_sum2}, we get that, almost surely as $n\uparrow\infty$,
\[\liminf_{n\to\infty}\frac1n \sum_{i=1}^n Q_{\xi(i)}(\mathbb Z^d)
\geq  \frac{\lambda}{\kappa} + \bigg(1-\frac{\lambda}{\kappa}\bigg) \frac{\lambda_0-\lambda-1}{\lambda_0-\lambda}.\]
Together with~\eqref{eq:first_fix}, this implies that (A'1) holds with
\begin{equation}\label{eq:def_c}
c = \frac{\lambda}{\kappa} + \bigg(1-\frac\lambda\kappa\bigg) \frac{\lambda_0-\lambda-1}{\lambda_0-\lambda}.
\end{equation}

\begin{remark}\label{remark:A'1} 
Because $\min_{x\in\mathbb Z^d} Q_x(\mathbb Z^d) = \lambda$, it would have been much easier to prove that \cite[Assumption (A1)]{MV20} holds with $c = \lambda/\kappa$, but this constant is not large enough for (A'2) to hold.
This is the reason why we need to first prove the weak localisation result in Theorem~\ref{prop:weakloc} (i) before using MVPPs to get a strong localisation result. 
\end{remark}
}

\subsection{Checking Assumption (A'2)}\label{sub:A2}
We let $V(x) = \|x\|_1$ if $x\neq {\bf 0}$, and $V(0)=1$.

(i) For all $N\geq 0$, $\{x\in\mathbb Z^d\colon \|x\|_1\leq N\}$ is finite and thus compact.

(ii) By Equation~\eqref{eq:Qx}, for all $x\in\mathbb Z^d$,
\[\kappa Q_x V 
= (\lambda_x-1)V(x) + \frac{1}{1+\lambda_x}\mathbb E[(1+\lambda_{x+\Delta})V(x+\Delta)].\]
In particular, for all $x\in\mathbb Z^d$ such that $\|x\|_1>1$,
\[\kappa Q_x V = (\lambda-1)V(x) +  \mathbb E[V(x+\Delta)].\]
For all $\|x\|_1>1$, we have $x+\Delta \neq {\bf 0}$ and thus, by the triangular inequality,
\[V(x+\Delta) = \|x+\Delta\|_1\leq \|x\|_1 + \|\Delta\|_1 = \|x\|_1+1.\]
Therefore,
\[\kappa Q_x V  
{\le} (\lambda-1)V(x) +  V(x)+1
= \lambda V(x) +1.\]
We thus get that, for all $x\in\mathbb Z^d$, 
\begin{equation}\label{eq:Lyap1}
\kappa Q_x V \leq \lambda V(x) + 1 + \max_{\|x\|_1\leq 1} \kappa Q_x V.
\end{equation}
Similarly, for all $\alpha\in (\nicefrac12, 1)$, for all $x\in\mathbb Z^d$ such that $\|x\|_1>1$,
\[\kappa Q_x V^\alpha 
= (\lambda-1)V(x)^\alpha + \mathbb E[V(x+\Delta)^\alpha].\]
Almost surely, for all $x\in\mathbb Z^d$ such that $\|x\|_1{>1}$, 
\[V(x+\Delta)^\alpha = \|x+\Delta\|^\alpha_1\leq (\|x\|_1 + 1)^\alpha.\]
As $\|x\|_1\uparrow\infty$, $(\|x\|_1 + 1)^\alpha\sim \|x\|_1^\alpha$.
We fix $\varepsilon>0$ small enough so that $(\lambda+\varepsilon)/\kappa<c$ 
(this is possible because $\lambda/\kappa<c$ - see~\eqref{eq:def_c}). 
Then, we choose $L(\varepsilon, \alpha)>1$ large enough so that $\|x\|_1\geq L(\varepsilon, \alpha)$ implies 
$(\|x\|_1 + 1)^\alpha\leq (1+\varepsilon)\|x\|_1^\alpha$.
With these choices, for all $x\in\mathbb Z^d$ such that $\|x\|_1\geq L(\varepsilon, \alpha)$,
\[V(x+\Delta)^\alpha = \|x+\Delta\|^\alpha_1\leq (1+\varepsilon) V(x)^\alpha,\]
and thus
\[\kappa Q_x V^\alpha 
\leq (\lambda-1)V(x)^\alpha +(1+\varepsilon) V(x)^\alpha
= (\lambda+\varepsilon) V(x)^\alpha.\]
This implies
\begin{equation}\label{eq:Lyap2}
\kappa Q_x V^\alpha 
\leq (\lambda+\varepsilon) V(x)^\alpha + \max_{\|x\|_1\leq L(\varepsilon)} V(x)^\alpha.
\end{equation}
By Equations~\eqref{eq:Lyap1} and~\eqref{eq:Lyap2}, we get that (ii) holds for any $q = \nicefrac1\alpha \in (1,2)$: indeed, one can take $\theta = (\lambda+\varepsilon)/\kappa<c$ and $K = \max_{\|x\|_1\leq L(\varepsilon,\alpha)} V(x)^\alpha/\kappa$.

(iii) Let $f : \mathbb Z^d \to \mathbb R$ be a function bounded by $1$.
First note that, for all $x\in\mathbb Z^d$, almost surely, by the triangular inequality and because $B_x\in \{0,1\}$,
\[|{\kappa}R^{\sss (1)}_x f| = |(2B_x -1)f(x) + (1-B_x)f(x+\Delta)|
\leq 2\|f\|_\infty\leq 2.\]
Similarly, 
\ba
|\kappa Q^{\sss (1)}_x f| 
&= |(2B_x -1)(1+\lambda_x)f(x) + (1-B_x)(1 +\lambda_{x+\Delta})f(x+\Delta)|\\
&\leq 2(1+\lambda_0)\|f\|_\infty \leq 2(1+\lambda_0).
\ea
Using Jensen's inequality, these imply that
\[|R_x \cdot f|\leq \mathbb E[|R^{\sss (1)}_x\cdot f|] \leq 2/\kappa
\quad\text{ and }\quad
|Q_x\cdot f|\leq \mathbb E[|Q^{\sss (1)}_x\cdot f|]\leq 2(1+\lambda_0)/\kappa.\]
Thus,
\[|Q_x\cdot f|^2 \leq \bigg(\frac{2(1+\lambda_0)}{\kappa}\bigg)^2 
\leq \frac{4(1+\lambda_0)^2}{\kappa^2}
\leq \frac{4(1+\lambda_0)^2}{\kappa^2}\cdot V(x),\]
because $V(x)\geq 1$ for all $x\in\mathbb Z^d$.
Now, for all $r>1$, using first the triangular inequality and then the fact that $x\mapsto x^r$ is convex on $[0,\infty)$ and thus $(a+b)^r\leq 2^{r-1}(a^r+b^r)$ for all $a,b\in[0,\infty)$, we get
\ba
\mathbb E \big[\big| {R}^{\sss (1)}_{x} \cdot f - R_{x}\cdot f\big|^r\big]
&\leq \mathbb E\big[\big(\big| {R}^{\sss (1)}_{x} \cdot f\big|+\big|R_{x}\cdot f\big|\big)^r\big]
\leq 2^{r-1} \big(\mathbb E\big[\big| {R}^{\sss (1)}_{x} \cdot f\big|^r\big]
+\big|R_{x}\cdot f\big|^r \big)\\
&\leq 2^{r} (2/\kappa)^r \leq (\nicefrac4\kappa)^r.
\ea
Thus, for any $r\in (1,2)$,
\[\mathbb E \big[\big| {R}^{\sss (1)}_{x} \cdot f - R_{x}\cdot f\big|^r\big]
\leq (1\vee \nicefrac4\kappa)^r \leq  (1\vee \nicefrac4\kappa)^2 V(x).\]
Similarly,
\[\mathbb E \big[\big| {Q}^{\sss (1)}_{x} \cdot f - Q_x\cdot f\big|^4\big]
\leq 2^{4}((1+\lambda_0)/\kappa)^4 
\leq 16((1+\lambda_0)/\kappa)^4 V(x).\]
Thus, (iii) holds for $q' = 2$, $p= 4$ and any $r\in (1,2)$. 

(iv) For all $q\geq1$,
for all $x\in\mathbb Z^d$, by the triangular inequality,
\begin{align*}
|\kappa Q^{\sss (1)}_x V^{\nicefrac1q}|^q
&= |(2B_x-1)(1+\lambda_x)V(x)^{\nicefrac1q}+(1-B_x)(1+\lambda_{x+\Delta})V(x+\Delta)|^q\\
&\leq (1+\lambda_0)^q(V(x)^{\nicefrac1q} + V(x+\Delta)^{\nicefrac1q})^q\\
&\leq 2^{q-1}(1+\lambda_0)^q(V(x) + V(x+\Delta)),
\end{align*}
because $q\geq1$ and thus $x\mapsto x^q$ is convex on $[0,\infty)$, 
implying that $(a+b)^q\leq 2^{q-1}(a^q+ b^q)$ for all $a,b\in[0,\infty)$.
We have proved before that, for all $x\in\mathbb Z^d$, $V(x+\Delta)\leq V(x)+1$ almost surely. 
This implies that, almost surely,
\begin{equation}\label{eq:bla1}
|\kappa Q^{\sss (1)}_x V^{\nicefrac1q}|^q\leq 2^{q-1}(1+\lambda_0)^q(2V(x)+1) 
\leq 2^{q+1}(1+\lambda_0)^qV(x),
\end{equation}
because $1\leq V(x)\leq 2V(x)$.
By Jensen's inequality, because $|\cdot|^q$ is convex,
\begin{equation}\label{eq:bla2}
|Q_x V^{\nicefrac1q}|^q
\leq \mathbb E[|Q^{\sss (1)}_x V^{\nicefrac1q}|^q]
\leq 2^{q+1}(1+\lambda_0)^qV(x)/\kappa^q.
\end{equation}
In particular, taking $q=1$ gives
\begin{equation}\label{eq:q=1}
|Q_x V|\leq 4(1+\lambda_0)V(x)/\kappa.
\end{equation}
For all $q\in (1,2)$, using the triangular inequality and the convexity of $x\mapsto x^q$ we get that, almost surely,
\[\left|Q^{\sss (1)}_x \cdot V^{\nicefrac1q}-Q_x\cdot V^{\nicefrac1q}\right|^q
\leq \big(\big|Q^{\sss (1)}_x \cdot V^{\nicefrac1q}\big|+\big|Q_x\cdot V^{\nicefrac1q}\big|\big)^q
\leq 2^{q-1}(\big|Q^{\sss (1)}_x \cdot V^{\nicefrac1q}\big|^q+\big|Q_x\cdot V^{\nicefrac1q}\big|^q).\]
Thus, by~\eqref{eq:bla1} and~\eqref{eq:bla2},
\[\left|Q^{\sss (1)}_x \cdot V^{\nicefrac1q}-Q_x\cdot V^{\nicefrac1q}\right|^q
\leq 2^q( 4(1+\lambda_0)/\kappa)^q V(x).\]
Thus, (iv) holds for any $r = q\in (1,2)$, 
and it thus holds for $r = q = \frac{p}{p-1} = \frac43$.

In total, we have showed that (A'2) holds for $q'=2$, $p=4$ and $q = r = \nicefrac43$.

\subsection{Checking Assumption (A'3)}\label{sub:A3}
To check Assumption (A'3), we apply a result of Champagnat and Villemonais (see~\cite[Theorem~5.1]{CV23}). 
We first give a statement of this result (this statement is a simplification of the original version in \cite{CV23}, which is enough for our purposes):
Let $(X(t))_{t\geq 0}$ be a continuous-time Markov process 
on a space~$E\cup\{\partial\}$, absorbed at $\partial$, 
with jump rates given by $(q_{x,y})_{x,y\in E\cup\{\partial\}}$ 
satisfying $\sum_{y\in E\cup\{\partial\}}q_{x,y}<\infty$ for all $x\in E$.
The infinitesimal generator of $(X(t))_{t\geq 0}$ acts on non-negative functions $f : E\cup\{\partial\} \to [0,\infty)$ satisfying $\sum_{y\in E\cup\{\partial\}} q_{x,y}f(y)<\infty$ for all $x\in E$, as follows: $\mathcal L f(\partial) = 0$ and, for all $x\in E$,
\[\mathcal Lf(x) = \sum_{y\in E\cup\{\partial\}} q_{x,y}(f(y)-f(x)).\]

\begin{theorem}[{\cite[Theorem~5.1]{CV23}}]\label{th:CV}
Assume that there exists a finite set $L\subset E$ 
such that $\mathbb P_x(X(1) = y) >0$ for all $x,y\in L$, and such that the constant 
\[\eta_2 
= \inf\{\eta>0 \colon \liminf_{t\uparrow\infty}\mathrm e^{\eta t} \mathbb P_x(X(t) = x)>0\}\]
is finite and does not depend on $x\in L$.
If there exists $C>0$, $\eta_1>\eta_2$ and $\varphi \colon E\cup \{\partial\}\to [0,\infty)$
such that $\varphi_{|E} \geq 1$, $\varphi(\partial) = 0$, $\sum_{y\in E\setminus\{x\}}q_{x,y} \varphi(y)<\infty$ for all $x\in E$, and
\[\mathcal L\varphi(x)\leq -\eta_1 \varphi(x)+ C{\bf 1}_{x\in L},\]
then
the process $(X(t))_{t\geq 0}$ admits a quasi-stationary distribution $\nu_{\mathit{QSD}}$, 
which is the unique one satisfying $\nu_{\mathit{QSD}}\varphi <\infty$ and $\mathbb P_{\nu_{\mathit{QSD}}}(X(t)\in L)>0$ for some $t\geq 0$.
Moreover, there exist constants $\alpha\in (0,1)$ and $K>0$ such that, for all probability measures $\mu$ on $E$ such that $\int\varphi\,\mathrm d\mu <\infty$ and $\mu(L)>0$,
\begin{equation}\label{eq:CV}
\|\mathbb P_\mu(X(t)\in \cdot | X(t)\neq \partial)- \nu_{\mathit{QSD}}\|_{\mathit{TV}}
\leq K \alpha^t \cdot \frac{\int\varphi\,\mathrm d\mu}{\mu(L)}.
\end{equation}
Furthermore, if we set $\varpi_t(x) = \mathrm e^{\eta_2 t}\mathbb P_x(X(t)\neq \partial)$, then there exists a function $\varpi : \mathbb E\mapsto [0,\infty)$ such that
\[\|\varpi_t-\varpi\|_{L^\infty(\varphi)} \to 0,\]
where 
\[L^{\infty}(\varphi) = \{f:E\mapsto R \colon \|f\|_{L^\infty(\varphi)}:=\sup_{x\in E} |f(x)|/\varphi(x)<\infty\}.\]
\end{theorem}

\begin{remark} 
Equation~\eqref{eq:CV} follows from taking $n_0 = 0$ in the definition of $\psi_2$ in~\cite[Theorem~5.1, see also Theorem~3.1]{CV23}; a private communication by Denis Villemonais confirmed to us that ``for all large enough $n_0$'' in the conclusion of~\cite[Theorem~3.1]{CV23} should be replaced by ``for any $n_0$''. In particular, the choice of $n_0 = 0$ is valid.
\end{remark}

We first prove that the assumptions of Theorem~\ref{th:CV} are satisfied:
Recall from~\eqref{eq:Qx} that, for all $x\in\mathbb Z^d$,
\begin{align*}
\kappa Q_x 
&= (\lambda_x-1)\delta_x + \frac{1}{1+\lambda_x}\mathbb E[(1+\lambda_{x+\Delta})\delta_{x+\Delta}]\\[5pt]
&=\begin{cases}
(\lambda_0-1)\delta_0 + \frac{1+\lambda}{1+\lambda_0}\mathbb E[\delta_{x+\Delta}] 
& \text{ if }x=0\\[5pt]
(\lambda-1)\delta_x + \mathbb E[\delta_{x+\Delta}]
+ \frac{1}{2d}\frac{\lambda_0-\lambda}{1+\lambda}\delta_0
& \text{ if }\|x\|_1=1\\[5pt]
 (\lambda-1)\delta_x +  \mathbb E[\delta_{x+\Delta}]
 & \text{ if }\|x\|_1>1.
\end{cases}
\end{align*}
Instead of considering the jump process $X$ of generator $(Q_x - \delta_x + (1-Q_x(\mathbb Z^d))\delta_\partial)_{x\in\mathbb Z^d}$, we look at the jump process $Y$ of generator $(\kappa Q_x - \kappa\delta_x + \kappa(1-Q_x(\mathbb Z^d))\delta_\partial)_{x\in\mathbb Z^d}$. 
One can couple $X$ and $Y$ so that $(X(\kappa t))_{t\geq 0} = (Y(t))_{t\geq0}$.

Using the definition of $\kappa$ (see~\eqref{eq:def_kappa}) and~\eqref{eq:QE}, we get that,
for all $x\in \mathbb Z^d$,
\begin{align*}
&\kappa Q_x - \kappa \delta_x +\kappa(1-Q_x(\mathbb Z^d))\delta_\partial\\
&= \begin{cases}
(\lambda_0-1-\kappa)\delta_0 
+ \frac{1+\lambda}{1+\lambda_0}\mathbb E[\delta_{x+\Delta}] 
& \text{ if }x=0\\[5pt]
(\lambda-1-\kappa)\delta_x + \mathbb E[\delta_{x+\Delta}]
+ \frac{1}{2d}\frac{\lambda_0-\lambda}{1+\lambda}\delta_0
+(\lambda_0-\lambda)\big(1-\frac1{1+\lambda_0}-\frac1{2d}\frac1{1+\lambda}\big)\delta_\partial
& \text{ if }\|x\|_1=1\\[5pt]
 (\lambda-1-\kappa)\delta_x +  \mathbb E[\delta_{x+\Delta}]
 +(\lambda_0-\lambda)\big(1-\frac1{1+\lambda_0}\big)\delta_\partial
 & \text{ if }\|x\|_1>1.
\end{cases}\\[5pt]
&= \begin{cases}
-\frac{1+\lambda}{1+\lambda_0}\delta_0 
+ \frac{1+\lambda}{1+\lambda_0}\mathbb E[\delta_{x+\Delta}] 
& \text{ if }x=0\\[5pt]
-\big(\lambda_0-\lambda+\frac{1+\lambda}{1+\lambda_0}\big)\delta_x 
+  \mathbb E[\delta_{x+\Delta}]
+ \frac{1}{2d}\frac{\lambda_0-\lambda}{1+\lambda}\delta_0
+(\lambda_0-\lambda)\big(1-\frac1{1+\lambda_0}-\frac1{2d}\frac1{1+\lambda}\big)\delta_\partial
& \text{ if }\|x\|_1=1\\[5pt]
-\big(\lambda_0-\lambda+\frac{1+\lambda}{1+\lambda_0}\big)\delta_x 
+  \mathbb E[\delta_{x+\Delta}]
 +(\lambda_0-\lambda)\big(1-\frac1{1+\lambda_0}\big)\delta_\partial
 & \text{ if }\|x\|_1>1.
\end{cases}
\end{align*}
In other words, the process $Y$ describes the movement on $\mathbb Z^d$ of a particle that behaves as follows:
\begin{itemize}
\item When the particle sits at 0, it jumps at rate $\frac{1+\lambda}{1+\lambda_0}$, and when it jumps, it moves to a neighbouring site chosen uniformly at random among the $2d$ possible choices.
\item When the particle sits at $x$ and $\|x\|_1 = 1$, it jumps to a uniformly-chosen neighbouring site at rate~$1$, it jumps to $0$ with an additional rate of $\frac{1}{2d}\frac{\lambda_0-\lambda}{1+\lambda}$, and it dies at rate $(\lambda_0-\lambda)\big(1-\frac1{1+\lambda_0}-\frac1{2d}\frac1{1+\lambda}\big)$.
\item When the particle sits at $x$ and $\|x\|_1 >1$, it jumps to a uniformly-chosen neighbouring site at rate~$1$ and it dies at rate $(\lambda_0-\lambda)\big(1-\frac1{1+\lambda_0}\big)$.
\end{itemize}
Note in particular that the dying rate is the largest when $\|x\|_1>1$ and the lowest at $x = 0$ (in fact, the particle does not die when it sits at the origin). Also note the additional drift towards zero when the particle sits at a neighbouring site of the origin.

To prove (A'3), we aim at applying Theorem~\ref{th:CV} and thus check that
$(Y(t))_{t\geq 0}$ satisfies its assumptions.
We fix
\begin{equation}\label{eq:def_rho1}
\rho_1
= (\lambda_0-\lambda)\Big(1-\frac1{1+\lambda_0}\Big) 
= \frac{\lambda_0(\lambda_0-\lambda)}{(1+\lambda_0)},
\end{equation}
and
\begin{equation}\label{eq:def_rho2}
\rho_2 = (\lambda_0 -\lambda)\Big(1-\frac1{1+\lambda_0}-\frac1{2d}\cdot\frac1{1+\lambda}\Big)+\frac{2d-1}{2d}.
\end{equation}
Note that $\rho_2<\rho_1$; indeed, this is equivalent to
\[(\lambda_0 -\lambda)\Big(1-\frac1{1+\lambda_0}-\frac1{2d}\cdot\frac1{1+\lambda}\Big)+\frac{2d-1}{2d}
<(\lambda_0-\lambda)\Big(1-\frac1{1+\lambda_0}\Big)\]
which is equivalent to~$2d-1+2d\lambda<\lambda_0$, which is our assumption.
We choose $\varepsilon>0$ small enough so that $\rho_2<\rho_1-\varepsilon$.
We then choose $K>\nicefrac1\varepsilon$ and let
$L = \{x\in\mathbb Z^d\colon \|x\|_1\leq K\}$. 
We also let $\varphi(x) = \max(1, \|x\|_1^{\nicefrac34})$, for all $x\in\mathbb Z^d$ (and $\varphi(\partial) = 0$, as required).

First note that, for any $x,y\in L$, and $t\geq 2$, by Markov's property,
\[\mathbb P_x(Y(t) = x)
\geq \mathbb P_x(Y(1) = y)\mathbb P_y(Y(t-2) = y)\mathbb P_y(Y(1) = x).\]
Thus, if we let $\frak m = \min_{x,y\in L}\mathbb P_x(Y(1) = y)$, then
\[\liminf_{t\uparrow\infty}\mathrm e^{\eta t}\mathbb P_x(X(t) = x) 
\geq \frak m^2 \mathrm e^{-2\eta} 
\liminf_{t\uparrow\infty}\mathrm e^{\eta t}\mathbb P_y(X(t) = y),
\]
which implies that $\eta_2$ does not depend on $x\in L$; it only remains to show that it is finite.
Now note that, if we let $L_0 = \{\|x\|\leq 1\}$ then, for all $x\in L$, $t\geq 2$,
\ban
\mathbb P_x(Y(t) = x)
&\geq \sum_{y, z\in L_0} \mathbb P_x(Y(1) = y)\mathbb P_y(Y(t-2) = z)\mathbb P_z(Y(1) = x)\notag\\
&\geq \frak m^2 \sum_{y,z\in L_0} \mathbb P_y(Y(t-2) = z)
= \frak m^2 \sum_{y\in L_0} \mathbb P_y(Y(t-2)\in L_0).\label{eq:L0}
\ean
Note that $\mathbb P_y(Y(t-2)\in L_0)\geq \mathbb P_y(Y(s)\in L_0\text{ for all }s\leq t-2)$.
Because the jump rate from a site $y$ such that $\|y\| = 1$ outside of $L_0$ is $\rho_2$ (see~\eqref{eq:def_rho2} for the definition of $\rho_2$), and because the process has to visit one of these sites before exiting $L_0$, we have that
\[\mathbb P_y(Y(s)\in L_0\text{ for all }s\leq t-2)\geq \mathrm e^{-\rho_2 (t-2)}.\]
This implies that, for all $x\in L$,
\[\mathbb P_x(Y(t) = x)
\geq \frak m^2 \mathrm e^{-\rho_2(t-2)},\]
which implies that $\liminf_{t\uparrow\infty} \mathrm e^{\rho_2 t}\mathbb P_x(Y(t) = x) \geq  \frak m^2 \mathrm e^{2\rho_2}>0$. Thus,
\[\eta_2 \leq \rho_2<\infty.\]

We now let $\mathcal L$ be the infinitesimal generator of $(Y(t))_{t\geq 0}$. Recall that $\varphi(x) = \max(1, \|x\|_1)$, for all $x\in\mathbb Z^d$. From now on, we write $\|\cdot\|$ in lieu of $\|\cdot\|_1$.
Because the death rate outside of $L$ equals $\rho_1$ (see~\eqref{eq:def_rho1} for the definition of $\rho_1$), we get that, for all $x\notin L$, for all $t\geq 0$,
\[\mathcal L \varphi(x) 
=\sum_{i=1}^d \frac{1}{2d} \big(\|x+e_i\|^{\nicefrac34}+\|x-e_i\|^{\nicefrac34}-2\|x\|^{\nicefrac34}\big)
- \rho_1 \|x\|^{\nicefrac34}.\] 
By symmetry, 
we may assume without loss of generality that $x\in \mathbb N_{\geq 0}^d$.
Under this assumption, $\|x+e_i\| = \|x\|+1$ for all $1\leq i\leq d$, and $\|x-e_i\| = \|x\|-1$ if $x_i\neq 0$, $\|x-e_i\| = \|x\|+1$ if $x_i = 0$. 
This gives that, for all $x\in\mathbb N_{\geq 0}^d$, for all $1\leq i\leq d$,
\[\|x+e_i\|^{\nicefrac34} + \|x-e_i\|^{\nicefrac34} -2\|x\|^{\nicefrac34}
\leq 2(\|x\|+1)^{\nicefrac34}-2\|x\|^{\nicefrac34},\]
which implies
\ba
\mathcal L \varphi(x) 
&\leq (\|x\|+1)^{\nicefrac34}-(1+\rho_1)\|x\|^{\nicefrac34} 
= \Big(\Big(1+\frac1{\|x\|}\Big)^{\nicefrac34}-1-\rho_1\Big) \|x\|^{\nicefrac34}\\
&\leq -(\rho_1 - \nicefrac1K) \|x\|^{\nicefrac34}
\leq -(\rho_1-\varepsilon)\varphi(x),
\ea
as long as $\|x\|>K$.
We thus get
\[\mathcal L \varphi(x) \leq -\eta_1 \varphi(x) + C{\bf 1}_{x\in L},\]
where we have set $C = \max_{x\in L} \mathcal L \varphi(x)<\infty$ and $\eta_1 = \rho_1-\varepsilon$. 
Recall that we have chosen $\varepsilon$ such that $\eta_2 \leq  \rho_2<\rho_1-\varepsilon = \eta_1$. 
We have thus shown that the assumptions of Theorem~\ref{th:CV} hold for $L = \{\|x\|\leq K\}$ for some $K$ large enough and $\varphi = \max(1, \|\cdot\|^{\nicefrac34}) = V^{\nicefrac1q}$ (for the choices of $V$ and $q$ made in Section~\ref{sub:A2}).
We thus get that $Y$ admits a quasi-stationary distribution $\nu$ and that, for any $\mu\in\mathcal P_C(\mathbb Z^d)$ such that $\|\mu-\nu\|_{TV}\le \nu(L)/2$,
\[\|\mathbb P_\mu(Y(t)\in \cdot | Y(t)\neq \partial)- \nu\|_{\mathit{TV}}
\leq C \alpha^t \cdot \frac{\int\varphi\,\mathrm d\mu}{\mu(L)}
\leq C \alpha^t \cdot \frac{C}{\nu(L)/2}.\]
Because $X(t) = Y(\nicefrac t\kappa)$, we get that
\[\|\mathbb P_\mu(X(t)\in \cdot | X(t)\neq\partial)- \nu\|_{\mathit{TV}}
= \|\mathbb P_\mu(Y(\nicefrac t\kappa)\in \cdot | Y(\nicefrac t\kappa)\neq\partial)- \nu\|_{\mathit{TV}}
\leq C \alpha^{t/\kappa} \cdot \frac{C}{\nu(L)/2}.\]
Thus, to prove that (A'3) holds, it only remains to prove that, for all $\alpha\in\mathcal P_C(\mathbb Z^d)$, $\mathbb P_\alpha(X(t)\neq\partial)\geq \mathrm e^{-(1-c_0)t}$, i.e.\ that
\begin{equation}\label{eq:but}
\mathbb P_\alpha(Y(t)\neq\partial)\geq \mathrm e^{-\kappa(1-c_0)t},
\end{equation}
for some constant $c_0>\theta = \nicefrac\lambda\kappa$.
If we set $\varpi_t(x) = \mathrm e^{\eta_2 t}\mathbb P_x(Y(t)\neq\partial)$ for all $t\geq 0$ and $x\in\mathbb Z^d$, then 
\[\mathrm e^{\eta_2 t}\mathbb P_\alpha(Y(t)\neq\partial) 
= \alpha\cdot \varpi \bigg(= \int_{\mathbb Z^d} \varpi_t \mathrm d\alpha\bigg).\]
We apply the dominated convergence theorem to this integral:
To do so, first note that, by Theorem~\ref{th:CV}, there exists a function $\varpi$ such that $\|\varpi_t-\varpi\|_{L^{\infty}(\varphi)}\to 0$ as $t\uparrow\infty$. This implies that, for all $\varepsilon>0$, there exists $t_\varepsilon$ such that, for all $t\geq t_\varepsilon$, for all $x\in\mathbb Z^d$,
\begin{equation}\label{eq:cvLinfty}
|\varpi_t(x)-\varpi(x)|\leq \varepsilon\varphi(x).
\end{equation}
Thus
\[\big|\alpha \cdot (\varpi_t-\varpi) \big|
\leq \varepsilon \alpha\cdot \varphi \leq C\varepsilon,\]
because $\alpha\in\mathcal P_C(\mathbb Z^d)$.
This implies
\[0\leq\alpha\cdot \varpi_t
= \alpha\cdot \varpi  +\alpha\cdot (\varpi_t-\varpi)
\leq \|\varphi\|_{L^\infty(\varphi)} \alpha\cdot \varphi + C\varepsilon
<\infty,\]
because $\varpi\in L^{\infty}(\varphi)$ (as a limit of functions in $L^{\infty}(\varphi)$) and $\alpha\in\mathcal P_C(\varphi)$.
Thus, by dominated convergence, we get that
\begin{equation}\label{eq:nearly}
\mathrm e^{\eta_2 t}\mathbb P_\alpha(Y(t)\neq\partial) 
= \alpha\cdot \varpi_t  \to 
\alpha\cdot \varpi.
\end{equation}
We now prove that $\alpha\cdot\varpi>0$; this holds because $\varpi(x)>0$ for all $x\in \mathbb Z^d$. Indeed, reasoning as in~\eqref{eq:L0}, we have that, for all $x\in\mathbb Z^d$, for all $t>2$,
\begin{align*}
\mathbb P_x(Y(t)\neq\partial)
&\geq \mathbb P_x(Y(t) = x)
\geq \sum_{z\in L_0} \mathbb P_x(Y(1) = 0) \mathbb P_0(Y(t-2) = z) \mathbb P_z(Y(1) = x)\\
&\geq \mathbb P_x(Y(1) = 0) \big(\min_{z\in L_0} \mathbb P_z(Y(1) = x)\big)
\mathbb P_0(Y(s)\in L_0 \text{ for all }0\leq s\leq t-2)\\
&\geq \mathbb P_x(Y(1) = 0) \big(\min_{z\in L_0} \mathbb P_z(Y(1) = x)\big) 
\mathrm e^{-\rho_2(t-2)}.
\end{align*}
Because $\rho_2 = \eta_2$,
\[\varpi(x) = \lim_{t\uparrow\infty} \mathrm e^{\eta_2 t} \mathbb P_x(Y(t)\neq\partial)
\geq \mathbb P_x(Y(1) = 0) \big(\min_{z\in L_0} \mathbb P_z(Y(1) = x)\big)\mathrm e^{2\eta_2}>0.\]
This concludes the proof that $\alpha\cdot\varpi>0$. 
Thus, \eqref{eq:nearly} implies that there exists $\varepsilon>0$ such that, for all $t\geq 0$,
\[\mathrm e^{\eta_2 t}\mathbb P_\alpha(Y(t)\neq\partial) > \varepsilon.\]
Thus, to prove~\eqref{eq:but}, it is enough to show that $\rho_2 = \eta_2<\kappa(1-\theta) = \kappa-\lambda$. By~\eqref{eq:def_kappa}, 
\[\kappa - \lambda = \frac{\lambda_0(\lambda_0-\lambda)}{1+\lambda_0}.\]
Also, by~\eqref{eq:def_rho2},
\[\eta_2 = \frac{\lambda_0(\lambda_0-\lambda)}{1+\lambda_0} 
+\frac{2d-1}{2d} - \frac{\lambda_0 -\lambda}{2d(1+\lambda)}.\]
Thus $\eta_2 < \kappa-\lambda$ if and only if
\[\frac{2d-1}{2d} < \frac{\lambda_0 -\lambda}{2d(1+\lambda)}
\quad\Leftrightarrow\quad
\lambda_0>2d-1+2d\lambda,\]
which holds, by assumption.

\paragraph{Acknowledgements:}
We are very grateful to Denis Villemonais for guiding us through~\cite{CV23}, Daniel Kious for early discussions on this project, Marcel Ortgiese for comments on the discussion of the literature, and Ofer Zeitouni for insightful comments, which led to Proposition~\ref{prop:exp}.

\bibliographystyle{alpha}
\bibliography{MABC}  
\end{document}